\documentclass[preprint,11pt]{elsarticle}

\usepackage{amsfonts, amsmath, amscd}
\usepackage[psamsfonts]{amssymb}

\usepackage{amssymb}

\usepackage{pb-diagram}

\usepackage[all,cmtip]{xy}

\usepackage[usenames]{color}

\newtheorem{theorem}{Theorem}[section]
\newtheorem{lemma}[theorem]{Lemma}
\newtheorem{corollary}[theorem]{Corollary}

\newtheorem{remark}[theorem]{Remark}
\newtheorem{proposition}[theorem]{Proposition}
\newtheorem{definition}[theorem]{Definition}
\newtheorem{example}[theorem]{Example}

\newtheorem{problem}[theorem]{Problem}

\newproof{proof}{Proof}

\numberwithin{equation}{section}
\numberwithin{theorem}{section}

\newcommand{\e}{\varepsilon}
\newcommand{\w}{\omega}

\newcommand{\IR}{\mathbb{R}}

\newcommand{\ff}{\mathbb{F}}

\newcommand{\IF}{\mathbb{F}}

\newcommand{\xxx}{\mathbf{x}}

\newcommand{\TTT}{\mathcal{T}}

\newcommand{\EE}{{\mathcal E}}

\newcommand{\Nn}{\mathcal{N}}

\newcommand{\LL}{\mathcal{L}}

\newcommand{\supp}{\mathrm{supp}}

\newcommand{\Ra}{\Rightarrow}

\newcommand{\cacx}{\overline{\mathrm{acx}}}

\newcommand{\Bo}{\mathsf{Bo}}

\newcommand{\Id}{\mathsf{id}}

\newcommand{\ind}{\mathsf{ind}}

\newcommand{\spn}{\mathsf{span}}
\newcommand{\cspn}{\overline{\mathsf{span}}}

\newcommand{\SI}{\underrightarrow{\mbox{ s-$\mathsf{ind}$}}_n\,}
\newcommand{\SM}{{\setminus}}

\begin{document}

\begin{frontmatter}

\title{Gelfand--Phillips type properties of locally convex spaces}

\author{Saak Gabriyelyan}
\ead{saak@math.bgu.ac.il}
\address{Department of Mathematics, Ben-Gurion University of the Negev, Beer-Sheva, P.O. 653, Israel}

\begin{abstract}
Let $1\leq p\leq q\leq\infty.$ Being motivated by the classical notions of the Gelfand--Phillips property and the (coarse) Gelfand--Phillips property of order $p$ of Banach spaces, we introduce and study different types of the Gelfand--Phillips property of order $(p,q)$ (the $GP_{(p,q)}$ property) and the  coarse Gelfand--Phillips property of order $p$ in the realm of all locally convex spaces.  We compare these classes and show that they are stable under taking direct product, direct  sums and closed subspaces. It is shown that any locally convex space is a quotient space of a locally convex space with the $GP_{(p,q)}$ property. Characterizations of locally convex spaces with the introduced Gelfand--Phillips type properties are given.
\end{abstract}


\end{frontmatter}


\section{Introduction}


All locally convex spaces  are assumed to be Hausdorff and over the field $\IF$ of real or complex numbers. We denote by $E'$ the topological dual of a locally convex space (lcs for short) $E$. The topological dual space $E'$ of $E$ endowed with the weak$^\ast$ topology $\sigma(E',E)$ and the strong topology $\beta(E',E)$ is denoted by $E'_{w^\ast}$ and $E'_\beta$, respectively.  For a bounded subset $B\subseteq E$ and a functional $\chi\in E'$, we put
\[
\|\chi\|_B:= \sup\big\{ |\chi(x)|:x\in B\cup\{0\}\big\}.
\]
If $E$ is a Banach space, we denote by $B_E$ the closed unit ball of $E$.

\begin{definition} \label{def:limited-Banach} {\em
A bounded subset $B$ of a Banach space $E$ is called {\em limited} if each weak$^\ast$ null sequence $\{ \chi_n\}_{n\in\w}$ in $E'$ converges to zero uniformly on $B$, that is $\lim_{n\to\infty} \|\chi_n\|_B =0$.}
\end{definition}

Let us recall the classical notion of Gelfand--Phillips spaces.
\begin{definition} \label{def:GP} {\em
A Banach space $E$ is said to have the {\em Gelfand--Phillips property} or is a {\em Gelfand--Phillips space} if every limited set in $E$ is relatively compact.}
\end{definition}

In \cite{Gelfand} Gelfand  proved that  every separable Banach space is Gelfand--Phillips. On the other hand, Phillips \cite{Phillips} showed that the non-separable Banach space $\ell_\infty$ is not Gelfand--Phillips.

Following Bourgain and Diestel  \cite{BourDies}, a bounded linear operator $T$ from a Banach space $L$ into $E$ is called {\em limited} if $T(B_L)$ is a limited subset of $E$. 
In \cite{Drewnowski}, Drewnowski noticed the next characterization of Banach spaces with the Gelfand--Phillips property.

\begin{theorem}[\cite{Drewnowski}] \label{t:Drew-GP}
For a Banach space $E$ the following assertions are equivalent:
\begin{enumerate}
\item[{\rm (i)}] $E$ is Gelfand--Phillips;
\item[{\rm (ii)}] every limited weakly null sequence in $E$ is norm null;
\item[{\rm (iii)}] every limited operator with range in $E$ is compact.
\end{enumerate}
\end{theorem}
\noindent This characterization of Gelfand--Phillips Banach spaces plays a crucial role in many arguments for establishing the Gelfand--Phillips  property in Banach spaces. 
The Gelfand--Phillips property was intensively studied in particular in \cite{CGP,Drewnowski,DrewEm,Schlumprecht-Ph,Schlumprecht-C,SinhaArora}. It follows from results of Schlumprecht \cite{Schlumprecht-Ph,Schlumprecht-C} that the Gelfand--Phillips property is not a three space property (see also Theorem 6.8.h in \cite{CG}). In our recent paper \cite{BG-GP-Banach}, we give several new characterizations of  Gelfand--Phillips Banach spaces.

Another direction for studying the  Gelfand--Phillips property is to characterize Gelfand--Phillips spaces that belong to some important classes of Banach spaces. In the next proposition, whose short proof is given in Corollary~2.2 of \cite{BG-GP-Banach}, we provide some of the most important and general results in this direction (for all relevant definitions used in what follows see Section \ref{sec:pre}).

\begin{theorem} \label{t:Banach-GP}
A Banach space $E$ is Gelfand--Phillips if  one of the following conditions holds:
\begin{enumerate}
\item[{\rm (i)}] {\rm(\cite[Cor.~2.2]{BG-GP-Banach})} the closed unit ball  $B_{E'}$ of the dual space $E'$ endowed with the weak$^\ast$ topology is selectively sequentially pseudocompact;
\item[{\rm (ii)}]  {\rm(\cite{Gelfand})} $E$ is separable;
\item[{\rm (iii)}] {\rm(\cite[Prop.~2]{CGP})} $E$ is  separably weak$^\ast$-extensible; 
\item[{\rm (iv)}]  {\rm(cf. \cite[Th.~2.2]{Drewnowski} and \cite[Prop.~2]{Schlumprecht-C})} the space $E'_{w^\ast}$ is selectively sequentially pseudocompact at some $E$-norming set $S\subseteq E'$;
\item[{\rm(v)}]  {\rm(\cite[Th.~4.1]{DrewEm})} $E=C(K)$ for some  compact selectively sequentially pseudocompact space $K$.
\end{enumerate}
\end{theorem}

The notion of a limited set in locally convex spaces was introduced by Lindstr\"{o}m and Schlumprecht in \cite{Lin-Schl-lim} and independently by Banakh and Gabriyelyan in  \cite{BG-GP-lcs}. Since limited sets in the sense of \cite{Lin-Schl-lim} are defined using {\em equicontinuity}, to distinguish both notions  we  called them in \cite{BG-GP-lcs} by $\EE$-limited sets:
\begin{definition} \label{def:limited-lcs} {\em
A subset   $A$ of a locally convex space $E$ is called
\begin{enumerate}
\item[{\rm(i)}] {\em $\EE$-limited } if $\|\chi_n\|_A\to 0$ for every equicontinuous weak$^\ast$ null sequence $\{\chi_n\}_{n\in\w}$ in  $E'$ (\cite{Lin-Schl-lim});
\item[{\rm(ii)}] {\em limited }  if $\|\chi_n\|_A\to 0$ for every weak$^\ast$ null sequence $\{\chi_n\}_{n\in\w}$ in  $E'$ (\cite{BG-GP-lcs}).
\end{enumerate} }
\end{definition}
\noindent It is clear that if $E$ is a $c_0$-barrelled space, then $A$ is limited  if and only if it is $\EE$-limited, and for Banach spaces both these  notions of limited sets coincide.

Lindstr\"{o}m and Schlumprecht  \cite{Lin-Schl-lim} proposed  the following extention of the Gelfand--Phillips property to the class of all locally convex spaces.
\begin{definition}[\cite{Lin-Schl-lim}] \label{def:EGP} {\em
A locally convex space $E$ is called an {\em $\EE$-Gelfand--Phillips space} or it has the {\em  $\EE$-Gelfand--Phillips property} if every $\EE$-limited set in $E$ is precompact.}
\end{definition}
\noindent Also in Definition \ref{def:EGP} we use the notation the ``$\EE$-Gelfand--Phillips property'' instead of  the ``Gelfand--Phillips property'' as in \cite{Lin-Schl-lim} to emphasize the {\em  equicontinuity} of weak$^\ast$ null sequences in $E'$ (as we noticed above this extra-condition can be omitted for $c_0$-barrelled spaces but not in general) and the fact that this notion can be naturally extended from the case of Banach spaces to the general case not in a unique  way.

Fr\'{e}chet spaces with the $\EE$-Gelfand--Phillips property were thoroughly studied by Alonso in \cite{Alonso}. The following operator characterization of $\EE$-Gelfand--Phillips spaces was obtained by Ruess in \cite{Ruess-F}.

\begin{theorem}[\cite{Ruess-F}] \label{t:EGP-operator}
A locally convex space $E$ is an $\EE$-Gelfand--Phillips space if and only if for each locally convex {\rm(}equivalently, Fr\'{e}chet or Banach{\rm)} space $L$, a subset $H$ of the $\epsilon$-product $E\epsilon L$ of $E$ and $L$ is precompact if and only if
\begin{enumerate}
\item[{\rm (i)}] $\{T(\chi): T\in H\}$ is precompact in $L$ for all $\chi\in E'$;
\item[{\rm (ii)}] for any equicontinuous, weak$^\ast$ null sequence $\{\chi_n\}_{n\in\w}$ in $E'$, it follows that $T(\chi_n)\to 0$ in $L$ uniformly over all $T\in H$.
\end{enumerate}
\end{theorem}

A generalization of the Gelfand--Phillips property to the class of all locally convex spaces can be done not in a unique way as in Definition \ref{def:EGP}. Being motivated by the notion of the Josefson--Nissenzweig property introduced and studied in \cite{BG-JNP} and as it is explained in \cite{BG-JNP}, the following notion can be considered as a version of the Josefson--Nissenzweig property which holds {\em everywhere}. We shall say that a subset $A$ of a locally convex space $E$ is {\em barrel-bounded}  or {\em barrel-precompact} if $A$ is  bounded or, respectively, precompact in the strong topology $\beta(E,E')$ on the space $E$.
\begin{definition} \label{def:b-GP-lcs} {\em
A locally convex space $E$ is said to have {\em the $b$-Gelfand--Phillips property} or else $E$ is a {\em $b$-Gelfand--Phillips space} if every limited barrel-bounded subset of $E$ is barrel-precompact.}
\end{definition}
It is easy to see that a Banach space $E$ has the Gelfand--Phillips property if and only if it is $b$-Gelfand--Phillips in the sense of Definition \ref{def:b-GP-lcs}
(where the prefix ``$b$-'' is added to emphasize the ``barreled'' nature of the notion);  and if $E$ is a barrelled space, then $E$ has the $b$-Gelfand--Phillips property if and only if it is an $\EE$-Gelfand--Phillips space or a Gelfand--Phillips space in the sense of Lindstr\"{o}m and  Schlumprecht \cite{Lin-Schl-lim}.

Let $E$ and $H$ be locally convex spaces. Denote by $\LL(E,H)$ the family of all operators (= continuous linear maps) from $E$ to $H$. Let $p\in[1,\infty]$. Recall that a sequence $\{x_n\}_{n\in\w}$ in $E$ is called {\em weakly $p$-summable} if  for every $\chi\in E'$ it follows that $(\langle\chi,x_n\rangle)\in \ell_p$ if $p\in[1,\infty)$ and $(\langle\chi,x_n\rangle)\in c_0$ if $p=\infty$. The families $\ell_p^w(E)$ and $c_0^w(E)$  of all weakly $p$-summable sequences in $E$ are vector spaces which admit natural locally convex vector topologies such that they are complete if so is $E$, for details see Section 19.4 in \cite{Jar} or Section 4 in \cite{Gab-Pel}.

Unifying the notion of unconditional convergent (u.c.) operator and the notion of completely continuous operators (i.e., they transform weakly null sequences into norm null), Castillo and S\'{a}nchez selected in \cite{CS} the class of $p$-convergent operators. An operator $T:X\to Y$ between Banach spaces $X$ and $Y$ is called {\em $p$-convergent} if it transforms weakly $p$-summable sequences into norm null sequences. Using this notion they introduced and study Banach spaces with the Dunford--Pettis property of order $p$ ($DPP_p$ for short)  for every $p\in[1,\infty]$. A Banach space $X$ is said to have the $DPP_p$ if every weakly compact operator from $X$ into a Banach space $Y$ is $p$-convergent.

The influential article  of Castillo and S\'{a}nchez \cite{CS} inspired an intensive study of $p$-versions of numerous geometrical properties of Banach spaces. In particular, the following $p$-versions of limitedness were introduced by Karn and Sinha \cite{KarnSinha} and Galindo and Miranda \cite{GalMir}.

\begin{definition} \label{def:small-bounded-p} {\em
Let $p\in[1,\infty]$, and let $X$ be a Banach space. A  bounded subset $A$ of $X$ is called
\begin{enumerate}
\item[{\rm(i)}] a {\em $p$-limited set} if
\[
\big(\sup_{a\in A} |\langle\chi_n,a\rangle|\big)\in \ell_p \;\; \Big(\mbox{or } \big(\sup_{a\in A} |\langle\chi_n,a\rangle|\big)\in c_0 \; \mbox{ if } p=\infty\Big)
\]
for every $(\chi_n)\in \ell_p^w \big(X^\ast,\sigma(X^\ast,X)\big)$ (or $(\chi_n)\in c_0^w\big(X^\ast,\sigma(X^\ast,X)\big)$ if $p=\infty$) (\cite{KarnSinha});
\item[{\rm(ii)}]  a {\em coarse $p$-limited set} if for every $T\in\LL(E,\ell_p)$ $($or $T\in\LL(E,c_0)$ if $p=\infty$$)$, the set $T(A)$ is relatively compact (\cite{GalMir}).
\end{enumerate}}
\end{definition}

Now,  it is natural to define the (coarse) $p$-version of the Gelfand--Phillips property of Banach spaces as follows.

\begin{definition} \label{def:pGP-Banach} {\em
Let $p\in[1,\infty]$. A Banach space $X$ is said to have
\begin{enumerate}
\item[{\rm(i)}] the {\em $p$-Gelfand--Phillips property} (the {\em $GP_p$ property}) if every limited weakly $p$-summable sequence in $X$ is norm null (\cite{FZ-p});
\item[{\rm(ii)}] the {\em coarse $p$-Gelfand--Phillips property} (the {\em coarse $GP_p$ property}) if every coarse $p$-limited set in $E$ is relatively compact (\cite{GalMir}).
\end{enumerate} }
\end{definition}
\noindent Taking into account Theorem \ref{t:Drew-GP} it is easy to see that a Banach space $X$ has the $GP_\infty$ property if and only if it has the $GP$ property. Numerous results concerning the Gelfand--Phillips type properties were obtained by Deghhani at al. \cite{DMD}, Fourie and Zeekoei \cite{FZ-16}, Ghenciu \cite{Ghenciu-pGP,Ghenciu-20,Ghenciu-23}, and  Galindo and Miranda \cite{GalMir}.

The classes of limited, $\EE$-limited, $p$-limited and coarse $p$-limited sets can be defined in any locally convex space as follows.
\begin{definition}[\cite{Gab-limited}] \label{def:p-limit-coarse-set} {\em
Let $1\leq p\leq q\leq\infty$. A non-empty subset $A$ of a locally convex space $E$  is called
\begin{enumerate}
\item[{\rm(i)}]  a  {\em $(p,q)$-limited set} (resp., {\em $(p,q)$-$\EE$-limited set}) if
\[
\Big(\|\chi_n\|_A\Big)\in \ell_q \; \mbox{ if $q<\infty$, } \; \mbox{ or }\;\; \|\chi_n\|_A\to 0 \; \mbox{ if $q=\infty$},
\]
for every (resp., equicontinuous) weak$^\ast$ $p$-summable sequence $\{\chi_n\}_{n\in\w}$ in  $E'$;
\item[{\rm(ii)}]  a {\em coarse $p$-limited set} if for every $T\in\LL(E,\ell_p) $ $($or $T\in\LL(E,c_0)$ if $p=\infty$), the set $T(A)$ is relatively compact.
\end{enumerate}}
\end{definition}
\noindent We denote by $\mathsf{L}_{(p,q)}(E)$ and $\mathsf{EL}_{(p,q)}(E)$ the family of all $(p,q)$-limited subsets and all $(p,q)$-$\EE$-limited subsets of $E$, respectively. $(p,p)$-($\EE$-)limited sets and $(\infty,\infty)$-($\EE$-)limited sets will be called simply {\em $p$-{\rm($\EE$-)}limited sets} and  ($\EE$-){\em limited sets}, respectively. The family of all coarse $p$-limited sets is denoted by $\mathsf{CL}_p(E)$. 



Taking into consideration that the classes of precompact, sequentially precompact, relatively sequentially compact and  relatively compact subsets of a locally convex space are distinct in general, one can naturally define the following types of the Gelfand--Phillips property generalizing the corresponding notions  from Definitions \ref{def:GP}, \ref{def:EGP}, \ref{def:b-GP-lcs} and \ref{def:pGP-Banach}.

\begin{definition}\label{def:property-pGP}{\em
Let $1\leq p\leq q\leq\infty$. A locally convex space $E$ is said to have
\begin{enumerate}
\item[{\rm(i)}] a {\em $(p,q)$-Gelfand--Phillips property } (a {\em precompact $(p,q)$-Gelfand--Phillips property}, a {\em sequential $(p,q)$-Gelfand--Phillips property} or a {\em sequentially precompact $(p,q)$-Gelfand--Phillips property}) if every $(p,q)$-limited set in $E$ is relatively compact (resp., precompact, relatively sequentially compact or sequentially precompact); for short, the {\em $GP_{(p,q)}$ property} (resp.,  the {\em $prGP_{(p,q)}$ property}, the {\em $sGP_{(p,q)}$ property} or the {\em $spGP_{(p,q)}$ property});
\item[{\rm(ii)}] a {\em $(p,q)$-equicontinuous Gelfand--Phillips property } (a {\em precompact $(p,q)$-equicontinuous Gelfand--Phillips property}, a {\em sequential $(p,q)$-equicontinuous Gelfand--Phillips property} or a {\em sequentially precompact $(p,q)$-equicontinuous Gelfand--Phillips property}) if every $(p,q)$-$\EE$-limited set in $E$ is relatively compact (resp., precompact, relatively sequentially compact or sequentially precompact); for short, the {\em $EGP_{(p,q)}$ property} (resp.,  the {\em $prEGP_{(p,q)}$ property}, the {\em $sEGP_{(p,q)}$ property} or the {\em $spEGP_{(p,q)}$ property});
\item[{\rm(iii)}] a {\em $b$-$(p,q)$-Gelfand--Phillips property } (a {\em $b$-$(p,q)$-equicontinuous Gelfand--Phillips property }) if every $(p,q)$-limited (resp., $(p,q)$-$\EE$-limited) set in $E$ is barrel-precompact; for short, the {\em $b$-$GP_{(p,q)}$ property} or the {\em $b$-$EGP_{(p,q)}$ property}, respectively;
\item[{\rm(iv)}] a {\em coarse $p$-Gelfand--Phillips property } (a {\em coarse precompact $p$-Gelfand--Phillips property}, a {\em coarse  sequential $p$-Gelfand--Phillips property} or a {\em coarse sequentially precompact $p$-Gelfand--Phillips property}) if every coarse $p$-limited set in $E$ is relatively compact (resp., precompact, relatively sequentially compact or sequentially precompact); for short, the {\em coarse  $GP_{p}$ property} (resp.,  the {\em coarse $prGP_{p}$ property}, the {\em coarse $sGP_p$ property} or the {\em coarse $spGP_p$ property}).
\end{enumerate}
In the case $q=p$  we shall write only one subscript $p$, and if $q=p=\infty$ the subscripts will be omitted (so that we  say that $E$ has the $GP_p$ {\em property} or the $GP$ {\em property} etc., respectively).   }
\end{definition}

The main purpose of the article is to study locally convex spaces with Gelfand--Phillips type properties introduced in Definition \ref{def:property-pGP}. Now we describe the content of the article.

In Section \ref{sec:pre}  we fix basic notions and prove some necessary results used in the article. It is a well known result due to Odell and Stegall (see \cite[p.~377]{Rosen-94}) that any limited set in a Banach space is weakly sequentially precompact. In Theorem \ref{t:inf-V*-wsc} we essentially generalize this result using the same idea. By the celebrated Rosenthal $\ell_1$-theorem, any bounded sequence in a Banach space $E$ has a weakly Cauchy subsequence if and only if  $E$ has no an isomorphic copy of $\ell_1$. The remarkable result of Ruess \cite{ruess} shows that an analogous statement holds true for much wider classes of locally convex spaces. These facts motive to introduce the {\em weak Cauchy subsequence property of order $p$} ($wCSP_p$ for short) for any $p\in[1,\infty]$, see Definition \ref{def:wCSPp}. In Proposition \ref{p:Lp-wCSPp} we prove that if $1<p<\infty$, then $\ell_p$ has the $wCSP_p$ if and only if $p\geq 2$. We also show that the class of locally convex spaces with the $wCSP_p$ is stable under taking countable products and arbitrary direct sums, see Proposition \ref{p:sum-wCSP}.

In Section \ref{sec:perm-GP-property} we select natural relationships between Gelfand--Phillips type properties introduced in Definition \ref{def:property-pGP}, see Lemmas \ref{l:GPp=EGPp}, \ref{l:GPp=EGPp-2} and \ref{l:GPp=EGPp-3} and Propositions \ref{p:pq-lim-coarse-p-lim} and \ref{p:pq-lim-coarse-p-lim}. As a corollary we show in Proposition \ref{p:equal-GP-strict-LF} that if  $E$ is an angelic, complete  and barrelled space  (for example, $E$ is a strict $(LF)$-space), then all Gelfand--Phillips type properties from (i)-(iii) of Definition \ref{def:property-pGP} coincide.
If the space $E$ has some of Schur type properties, then $V^\ast$ type properties of Pe{\l}czy\'{n}ski' imply corresponding Gelfand--Phillips type properties, see Proposition \ref{p:V*pq=>GPpq}. It is easy to see (Lemma \ref{l:GPp=EGPp}) that the Gelfand--Phillips type properties of order $(1,\infty)$ are the strongest ones in the sense that if $E$ has for example the $spGP_{(1,\infty)}$ property, then it has the $spGP_{(p,q)}$ property for all $1\leq p\leq q\leq \infty$. Therefore one can expect that spaces with the $spGP_{(1,\infty)}$ property have some additional properties. This is indeed so, in  Proposition  \ref{p:wsGP-Schur-con} we show that if $E$ is a sequentially complete barrelled locally convex space whose bounded subsets are weakly sequentially precompact, then $E$ has the $spGP_{(1,\infty)}$ property if and only if it has the Schur property.
In Proposition \ref{p:product-sum-GP} we show that the classes of locally convex spaces with Gelfand--Phillips type properties are stable under taking direct products and direct sums, and Proposition \ref{p:subspace-pGP} stands  heredity properties of Gelfand--Phillips type properties. In Theorem \ref{t:Cp-EGPp} we characterize spaces $C_p(X)$ which have one of the equicontinuous $GP$ properties from (ii) of Definition \ref{def:property-pGP}. In Theorem \ref{t:L(X)-GPpq} we show that any lcs is a quotient space of an lcs with the $GP_{(p,q)}$ property; consequently, the class of locally convex spaces with the $GP_{(p,q)}$ property is sufficiently rich.

Characterizations of locally convex spaces with Gelfand--Phillips type properties are given in Section \ref{sec:char-GP-property}.
In Theorem \ref{t:sequential-GPpq-property} we obtain an operator characterization of locally convex spaces with the sequential $GP_{(p,q)}$ property and the sequentially precompact $GP_{(p,q)}$ property; our results generalize a characterization of Banach spaces with the Gelfand--Phillips property of order $p$ obtained by Ghenciu in  Theorem 15 of \cite{Ghenciu-pGP}. In Theorem \ref{t:sGPp-limited-p-conv} we characterize locally convex spaces with the sequentially precompact $GP_{(p,q)}$ property in an important class of locally convex spaces which includes all strict $(LF)$-spaces. In Theorem \ref{t:char-coarse-prGPp} we essentially generalize the equivalence (i)$\Leftrightarrow$(ii) in Theorem \ref{t:Drew-GP}. A natural direct extension of the Gelfand--Phillips property for Banach spaces is the precompact $GP_{(p,\infty)}$ property with $p\in[1,\infty]$ (so that the precompact Gelfand--Phillips property and the precompact $GP_{(1,\infty)}$ property are two extreme cases). Locally convex spaces with the precompact $GP_{(p,\infty)}$ property are characterized in Theorem \ref{t:GPpq-characterization}. Sufficient conditions to have the precompact $GP_{(p,\infty)}$ property are given in Theorem \ref{t:wGP-Mackey-c0} and Corollary \ref{c:precompact-GP}. Banach spaces with the coarse $p$-Gelfand--Phillips property are characterized by Galindo and Miranda  in Proposition 13 of \cite{GalMir}. In Theorem \ref{t:coarse-GPp-char} we generalize their result to all locally convex spaces.

Recall that a Banach space $E$ has the Gelfand--Phillips property if and only if every limited weakly null sequence in $E$ is norm null, see Theorem \ref{t:Drew-GP} (or the more general Theorem \ref{t:char-coarse-prGPp} below). This characterization motivates to introduce the following type of Gelfand--Phillips property which is studied in Section \ref{sec:GPscP}.
\begin{definition} \label{def:p-GP-q} {\em
Let $p,q,q'\in[1,\infty]$,  $q'\leq q$. A locally convex space $(E,\tau)$ is said to have the {\em $p$-Gelfand--Phillips sequentially compact property of order $(q',q)$} ($p$-$GPscP_{(q',q)}$ for short) if every weakly $p$-summable sequence in $E$ which is also a $(q',q)$-limited set is $\tau$-null. If $q'=q$, $q'=q=\infty$, $p=\infty$ or $p=q'=q=\infty$, we shall say simply that $E$ has the $p$-$GPscP_{q}$, the $p$-$GPscP$, $GPscP_{(q',q)}$ or the $GPscP$, respectively.}
\end{definition}
We characterize locally convex spaces $E$ with the $p$-$GPscP_{(q',q)}$ in Theorem \ref{t:char-p-GP}. This theorem extends and generalizes characterizations of Banach spaces with the Gelfand--Phillips property obtained in Theorem 2.2 of \cite{SaMo} and Corollary 13(ii) of \cite{Ghenciu-20}.
Spaces with the $GPscP_{(q,\infty)}$ are characterized in Theorem \ref{t:char-p-GPsc}.

In Section \ref{sec:strong-V*} we introduce strong versions of $V^\ast$ type properties of Pe{\l}czy\'{n}ski, which usually are stronger than the corresponding Gelfand--Phillips type properties. These new classes of locally convex spaces have nice stability properties, see Propositions \ref{p:product-sum-V*-strong} and \ref{p:subspace-V*-strong}. Locally convex spaces with the strong (resp., sequentially) precompact $V^\ast_p$ property are characterized in Theorem \ref{t:char-strong-V*p}.


\section{Preliminaries results} \label{sec:pre}



We start with some necessary definitions and notations used in the article. Set $\w:=\{ 0,1,2,\dots\}$. The closed unit ball of the field $\ff$ is denoted by $\mathbb{D}$. All topological spaces are assumed to be Tychonoff (= completely regular and $T_1$). The closure of a subset $A$ of a topological space $X$ is denoted by $\overline{A}$. The space $C(X)$ of all continuous functions on $X$ endowed with the pointwise topology is denoted by $C_p(X)$.
A Tychonoff space $X$  is called {\em Fr\'{e}chet--Urysohn} if for any cluster point $a\in X$ of a subset $A\subseteq X$ there is a sequence $\{ a_n\}_{n\in\w}\subseteq A$ which converges to $a$. A Tychonoff space $X$ is called an {\em angelic space} if (1) every relatively countably compact subset of $X$ is relatively compact, and (2) any compact subspace of $X$ is Fr\'{e}chet--Urysohn. Note that any subspace of an angelic space is angelic, and a subset $A$ of an angelic space $X$ is compact if and only if it is countably compact if and only if $A$ is sequentially compact, see Lemma 0.3 of \cite{Pryce}.


Let $E$ be a locally convex space. The span of a subset $A$ of $E$ and its closure are denoted by $E_A:=\spn(A)$ and $\cspn(A)$, respectively. We denote by $\Nn_0(E)$ (resp., $\Nn_{0}^c(E)$) the family of all (resp., closed absolutely convex) neighborhoods of zero of $E$. The family of all bounded subsets of $E$ is denoted by $\Bo(E)$. The value of $\chi\in E'$ on $x\in E$ is denoted by $\langle\chi,x\rangle$ or $\chi(x)$. A sequence $\{x_n\}_{n\in\w}$ in $E$ is said to be {\em Cauchy} if for every $U\in\Nn_0(E)$ there is $N\in\w$ such that $x_n-x_m\in U$ for all $n,m\geq N$. If $E$ is a normed space, $B_E$ denotes the closed unit ball of $E$. The family of all operators from $E$ to an lcs $L$ is denoted by $\LL(E,L)$.

For an lcs $E$, we denote by $E_w$ and $E_\beta$ the space $E$ endowed with the weak topology $\sigma(E,E')$ and with the strong topology $\beta(E,E')$, respectively. The topological dual space $E'$ of $E$ endowed with weak$^\ast$ topology $\sigma(E',E)$ or with the strong topology $\beta(E',E)$ is denoted by $E'_{w^\ast}$ or $E'_\beta$, respectively. The {\em polar} of a subset $A$ of $E$ is denoted by
$
A^\circ :=\{ \chi\in E': \|\chi\|_A \leq 1\}. 
$
A subset $B$ of $E'$ is {\em equicontinuous} if $B\subseteq U^\circ$ for some $U\in \Nn_0(E)$.

A subset $A$ of a locally convex space $E$ is called
\begin{enumerate}
\item[$\bullet$] {\em precompact} if for every $U\in\Nn_0(E)$ there is a finite set $F\subseteq E$ such that $A\subseteq F+U$;
\item[$\bullet$] {\em sequentially precompact} if every sequence in $A$ has a Cauchy subsequence;
\item[$\bullet$] {\em weakly $($sequentially$)$ compact} if $A$ is (sequentially) compact in $E_w$;
\item[$\bullet$] {\em relatively weakly compact} if its weak closure $\overline{A}^{\,\sigma(E,E')}$ is compact in $E_w$;
\item[$\bullet$] {\em relatively weakly sequentially compact} if each sequence in $A$ has a subsequence weakly converging to a point of $E$;
\item[$\bullet$] {\em weakly sequentially precompact} if each sequence in $A$ has a weakly Cauchy subsequence.
\end{enumerate}
Note that each sequentially precompact subset of $E$ is precompact, but the converse is not true in general, see Lemma 2.2 of \cite{Gab-Pel}. 
We shall use repeatedly the next lemma, see Lemma 4.4 in \cite{Gab-DP}.
\begin{lemma} \label{l:null-seq}
Let $\tau$ and $\TTT$ be two locally convex vector topologies on a vector space $E$ such that $\tau\subseteq \TTT$. If $S=\{x_n\}_{n\in\w}$ is $\tau$-null and $\TTT$-precompact, then $S$ is $\TTT$-null. Consequently, if $S$ is weakly $\TTT$-null and $\TTT$-precompact, then $S$ is $\TTT$-null.
\end{lemma}


In what follows we shall actively use the following classical completeness type properties and weak barrelledness conditions. A locally convex space  $E$ is
\begin{enumerate}
\item[$\bullet$]  {\em quasi-complete} if each closed bounded subset of $E$ is complete;
\item[$\bullet$]  {\em sequentially complete} if each Cauchy sequence in $E$ converges;
\item[$\bullet$]  {\em locally complete} if the closed absolutely convex hull of a null sequence in $E$ is compact;
\item[$\bullet$] ({\em quasi}){\em barrelled} if every $\sigma(E',E)$-bounded (resp., $\beta(E',E)$-bounded) subset of $E'$ is equicontinuous;
\item[$\bullet$] {\em $c_0$-}({\em quasi}){\em barrelled} if every $\sigma(E',E)$-null (resp., $\beta(E',E)$-null) sequence is equicontinuous.
\end{enumerate}
It is well known that $C_p(X)$ is quasibarrelled for every Tychonoff space $X$.

Denote by $\bigoplus_{i\in I} E_i$ and $\prod_{i\in I} E_i$  the locally convex direct sum and the topological product of a non-empty family $\{E_i\}_{i\in I}$ of locally convex spaces, respectively. If $0\not= \xxx=(x_i)\in \bigoplus_{i\in I} E_i$, then the set $\supp(\xxx):=\{i\in I: x_i\not= 0\}$ is called the {\em support} of $\xxx$. The {\em support}  of a subset $A$, $\{0\}\subsetneq A$, of $\bigoplus_{i\in I} E_i$ is the set $\supp(A):=\bigcup_{a\in A} \supp(a)$. We shall also consider elements $\xxx=(x_i) \in \prod_{i\in I} E_i$ as functions on $I$ and write $\xxx(i):=x_i$.

Denote by $\ind_{n\in \w} E_n$  the inductive limit of a (reduced) inductive sequence $\big\{(E_n,\tau_n)\big\}_{n\in \w}$  of locally convex spaces. If in addition $\tau_m{\restriction}_{E_n} =\tau_n$ for all $n,m\in\w$ with $n\leq m$, the inductive limit $\ind_{n\in \w} E_n$ is called {\em strict} and is denoted by $\SI E_n$. In the partial case when all spaces $E_n$ are Fr\'{e}chet, the strict inductive limit is called a {\em strict $(LF)$-space}.

Below we recall some of the basic classes of compact-type operators.
\begin{definition} \label{def:operators} {\em
Let $E$ and $L$ be locally convex spaces. An operator $T\in \LL(E,L)$ is called {\em compact} (resp., {\em sequentially compact, precompact, sequentially precompact, weakly compact, weakly sequentially compact, weakly sequentially precompact, bounded}) if there is $U\in \Nn_0(E)$ such that $T(U)$ a relatively compact (relatively sequentially compact,  precompact, sequentially precompact,  relatively weakly compact, relatively weakly sequentially compact,  weakly sequentially precompact or bounded) subset of $E$.}
\end{definition}

Let $p\in[1,\infty]$. Then $p^\ast$ is defined to be the unique element of $ [1,\infty]$ which satisfies $\tfrac{1}{p}+\tfrac{1}{p^\ast}=1$. For $p\in[1,\infty)$, the space $\ell_{p^\ast}$  is the dual space of $\ell_p$. We denote by $\{e_n\}_{n\in\w}$ the canonical basis of $\ell_p$, if $1\leq p<\infty$, or the canonical basis of $c_0$, if $p=\infty$. The canonical basis of $\ell_{p^\ast}$ is denoted by $\{e_n^\ast\}_{n\in\w}$. In what follows we usually identify $\ell_{1^\ast}$ with $c_0$. Denote by  $\ell_p^0$ and $c_0^0$ the linear span of $\{e_n\}_{n\in\w}$  in  $\ell_p$ or in $c_0$ endowed with the induced norm topology, respectively. We shall use also the following well known description of relatively compact subsets of $\ell_p$ and $c_0$,  see \cite[p.~6]{Diestel}.
\begin{proposition} \label{p:compact-ell-p}
{\rm(i)} A bounded subset $A$ of $\ell_p$, $p\in[1,\infty)$, is relatively compact if and only if
\[
\lim_{m\to\infty} \sup\Big\{ \sum_{m\leq n} |x_n|^p : x=(x_n)\in A\Big\} =0.
\]
{\rm(ii)} A bounded subset $A$ of $c_0$ is relatively compact if and only if
$
\lim_{n\to\infty} \sup\{ |x_n|: x=(x_n)\in A\} =0.
$
\end{proposition}


Let $p\in[1,\infty]$. A sequence  $\{x_n\}_{n\in\w}$ in a locally convex space $E$ is called
\begin{enumerate}
\item[$\bullet$] {\em weakly $p$-convergent to $x\in E$} if  $\{x_n-x\}_{n\in\w}$ is weakly $p$-summable;
\item[$\bullet$] {\em weakly $p$-Cauchy} if for each pair of strictly increasing sequences $(k_n),(j_n)\subseteq \w$, the sequence  $(x_{k_n}-x_{j_n})_{n\in\w}$ is weakly $p$-summable.
\end{enumerate}
A sequence  $\{\chi_n\}_{n\in\w}$ in $E'$ is called  {\em weak$^\ast$ $p$-summable} (resp., {\em weak$^\ast$ $p$-convergent to $\chi\in E'$} or  {\em weak$^\ast$ $p$-Cauchy})  if it is weakly $p$-summable (resp., weakly $p$-convergent to $\chi\in E'$ or weakly $p$-Cauchy) in $E'_{w^\ast}$.
Following \cite{Gab-Pel}, $E$ is called  {\em $p$-barrelled } (resp.,  {\em $p$-quasibarrelled}) if every weakly $p$-summable sequence in $E'_{w^\ast}$ (resp., in  $E'_\beta$) is equicontinuous.

Generalizing the corresponding notions in the class of Banach spaces introduced in \cite{CS} and \cite{Ghenciu-pGP}, the following $p$-versions of weakly compact-type properties are defined in \cite{Gab-Pel}. Let $p\in[1,\infty]$. A subset   $A$ of a locally convex space $E$ is called
\begin{enumerate}
\item[$\bullet$] ({\em relatively}) {\em weakly sequentially $p$-compact} if every sequence in $A$ has a weakly $p$-convergent  subsequence with limit in $A$ (resp., in $E$);
\item[$\bullet$] {\em weakly  sequentially $p$-precompact} if every sequence from $A$ has a  weakly $p$-Cauchy subsequence.
\end{enumerate}
An operator $T:E\to L$ between locally convex spaces $E$ and $L$ is called {\em weakly sequentially } $p$-({\em pre}){\em compact} if there is $U\in\Nn_0(E)$ such that $T(U)$ is relatively weakly sequentially $p$-compact (resp., weakly  sequentially $p$-precompact) subset of $L$.

The following class of subsets of an lcs $E$ was introduced and studied in \cite{Gab-Pel}, it generalizes the notion of $p$-$(V^\ast)$ subsets  of Banach spaces defined in \cite{CCDL}.
\begin{definition}\label{def:tvs-V*-subset} {\em
Let $p,q\in[1,\infty]$. A non-empty subset   $A$ of a locally convex space $E$ is called a  {\em $(p,q)$-$(V^\ast)$ set} (resp., a {\em $(p,q)$-$(EV^\ast)$ set}) if
\[
\Big(\sup_{a\in A} |\langle \chi_n, a\rangle|\Big)\in \ell_q \; \mbox{ if $q<\infty$, } \; \mbox{ or }\;\; \Big(\sup_{a\in A} |\langle \chi_n, a\rangle|\Big)\in c_0 \; \mbox{ if $q=\infty$},
\]
for every  (resp., equicontinuous) weakly $p$-summable sequence $\{\chi_n\}_{n\in\w}$ in  $E'_\beta$. $(p,\infty)$-$(V^\ast)$ sets and $(1,\infty)$-$(V^\ast)$ sets will be called simply {\em $p$-$(V^\ast)$ sets} and {\em $(V^\ast)$ sets}, respectively.} 
\end{definition}

Let us recall $V^\ast$ type properties of  Pe{\l}czy\'{n}ski defined in \cite{Gab-Pel}.
\begin{definition}\label{def:property-Vp*}{\em
Let $1\leq p\leq q\leq\infty$. A locally convex space $E$ is said to have
\begin{enumerate}
\item[$\bullet$] the {\em property $V^\ast_{(p,q)}$} (a {\em property $EV^\ast_{(p,q)}$}) if every $(p,q)$-$(V^\ast)$ (resp., $(p,q)$-$(EV^\ast)$) set in $E$ is relatively weakly compact;
\item[$\bullet$] the {\em property $sV^\ast_{(p,q)}$} (a {\em property $sEV^\ast_{(p,q)}$}) if every $(p,q)$-$(V^\ast)$ (resp., $(p,q)$-$(EV^\ast)$) set in $E$ is relatively weakly sequentially compact;
\item[$\bullet$] the {\em weak property $sV_{(p,q)}^\ast$} or a {\em property $wsV_{(p,q)}^\ast$}  (resp., a {\em weak property $sEV_{(p,q)}^\ast$} or a {\em property $wsEV_{(p,q)}^\ast$}) if each $(p,q)$-$(V^\ast)$-subset (resp.,  $(p,q)$-$(EV^\ast)$-subset) of $E$ is weakly sequentially precompact.
\end{enumerate}
In the case when $q=\infty$ we shall omit the subscript $q$ and say that $E$ has the  {\em property $V^\ast_{p}$} etc., and in the case when $q=\infty$ and $p=1$ we shall say that $E$ has the  {\em property $V^\ast$} etc. }
\end{definition}

Let $p\in[1,\infty]$. The $p$-Schur property of Banach spaces was defined in \cite{DM} and  \cite{FZ-pL}. Generalizing this notion and following \cite{Gab-Pel}, an lcs $E$ is said to have a {\em $p$-Schur property} if every weakly $p$-summable sequence is a null-sequence. In particular, $E$ has the Schur property if and only if it is an $\infty$-Schur space. Following \cite{Gab-DP}, $E$ is called a {\em weakly sequentially $p$-angelic space} if the family of all relatively weakly sequentially $p$-compact sets in $E$ coincides with the family of all relatively weakly compact subsets of $E$. The space $E$ is a {\em weakly $p$-angelic space}  if it is  a weakly sequentially $p$-angelic space and each weakly compact subset  of $E$ is Fr\'{e}chet--Urysohn.

Following \cite{Gabr-free-resp}, a sequence $A=\{ a_n\}_{n\in\w}$ in an lcs $E$ is said to be {\em equivalent to the standard unit basis $\{ e_n: n\in\w\}$ of $\ell_1$} if there exists a linear topological isomorphism $R$ from $\cspn(A)$ onto a subspace of $\ell_1$ such that $R(a_n)=e_n$ for every $n\in\w$ (we do not assume that the closure $\cspn(A)$ of the  $\spn(A)$ of $A$ is complete or that $R$ is onto). We shall say also that $A$ is an {\em $\ell_1$-sequence}.
Following \cite{GKKLP}, a locally convex space $E$ is said to have the {\em Rosenthal property} if every bounded sequence in $E$ has a subsequence which either (1) is Cauchy in the weak topology, or (2) is equivalent to the unit basis of $\ell_1$. The following remarkable extension of the celebrated Rosenthal $\ell_1$-theorem  was proved by Ruess \cite{ruess}: {\em each locally complete locally convex space $E$ whose every separable bounded set is metrizable has the Rosenthal property.} Thus every strict $(LF)$-space has the Rosenthal property.

Recall that an lcs $X$ is called {\em injective} if for every subspace $H$ of a locally convex space $E$, each operator $T:H\to X$ can be extended to an operator ${\bar T}: E\to X$.

In \cite[p.~377]{Rosen-94} Rosenthal pointed out a theorem of Odell and Stegall which states that any $\infty$-$(V^\ast)$ set of a Banach space is weakly sequentially precompact. In what follows we use the following generalization of this remarkable result which is of independent interest (our detailed proof follows the Odell--Stegall idea, cf. also Theorem 
5.21 of \cite{Gab-limited}).

\begin{theorem} \label{t:inf-V*-wsc}
Let $2\leq p\leq q\leq\infty$, and let $E$ be a locally convex space with the Rosenthal property. Then every $(p,q)$-$(V^\ast)$ subset of $E$ is  weakly sequentially precompact. Consequently, each $(p,q)$-limited subset of $E$ is  weakly sequentially precompact.
\end{theorem}

\begin{proof}
Suppose for a contradiction that there is a $(p,q)$-$(V^\ast)$ subset $A$ of $E$ which is not weakly sequentially precompact. So there is a sequence $S=\{x_n\}_{n\in\w}$ in $A$ which does not have a weakly Cauchy subsequence. By the Rosenthal property of $E$ and passing to a subsequence if needed, we can assume that $S$ is an $\ell_1$-sequence. Set $H:=\cspn(S)$, and let $P:H\to \ell_1$ be a topological isomorphism of $H$ onto a subspace of $\ell_1$ such that $P(x_n)=e_n$ for every $n\in\w$ (where $\{e_n\}_{n\in\w}$ is the standard unit basis of $\ell_1$).  Let $J:\ell_1\to \ell_p$, $I_1:\ell_1\to \ell_2$, and $I_2:\ell_2\to \ell_p$ be the natural inclusions; so  $J=I_2\circ I_1$. By the Grothendieck Theorem 1.13 of \cite{DJT}, the operator $I_1$ is $1$-summing. By the Ideal Property 2.4 of \cite{DJT}, $J$ is also $1$-summing, and hence, by the Inclusion Property 2.8 of \cite{DJT}, the operator $J$ is $2$-summing.
By the discussion after Corollary 2.16 of \cite{DJT}, the operator $J$ has a factorization
\[
\xymatrix{
J: \; \ell_1  \ar[r]^R  & L_\infty(\mu) \ar[r]^{J_2^\infty}  & L_2(\mu) \ar[r]^Q & \ell_p },
\]
where $\mu$ is a regular probability measure on some compact space $K$ and $J_2^\infty:L_\infty(\mu) \to L_2(\mu)$  is the natural inclusion. By Theorem 4.14 of \cite{DJT}, the Banach space $L_\infty(\mu)$ is injective. Therefore, by Lemma 
5.20 of \cite{Gab-limited}, $L_\infty(\mu)$ is  an injective locally convex space. In particular, the operator $R\circ P:H\to L_\infty(\mu) $ can be extended to an operator $T_\infty:E\to L_\infty(\mu) $. Set $T:=Q\circ  J_2^\infty \circ T_\infty$. Then $T$ is an operator from $E$ to $\ell_p$ such that
\[
T(x_n)=Q\circ  J_2^\infty \circ R\circ P(x_n)=J\circ P(x_n)=e_n \quad \mbox{ for every $n\in\w$},
\]
where $\{e_n\}_{n\in\w}$ is the standard unit basis of $\ell_p$. Since $A$ and hence also $S$ are $(p,q)$-$(V^\ast)$ sets, it follows that the canonical basis $\{e_n\}_{n\in\w}$ of $\ell_p$ is also a $(p,q)$-$(V^\ast)$ set. However this is impossible because the standard unit basis $\{e_n^\ast\}_{n\in\w}$ of the dual $\ell_{p^\ast}$ is weakly $p$-summable (see Example 
4.4 of \cite{Gab-Pel}), however since $\sup\{|\langle e_n^\ast,e_i\rangle|:i\in\w\}=1$ for all $n\in\w$, it follows that  $\{e_n\}_{n\in\w}$ is not a $(p,q)$-$(V^\ast)$ set.

The last assertion follows from the easy fact that any $(p,q)$-limited set is a $(p,q)$-$(V^\ast)$ set. \qed
\end{proof}

Being motivated by the celebrated Rosenthal's $\ell_1$ theorem, we introduce the following class of locally convex spaces.
\begin{definition} \label{def:wCSPp} {\em
Let $p\in[1,\infty]$. A locally convex space $E$ is said to have the {\em weak Cauchy subsequence property of order $p$} (the $wCSP_p$ for short) if every bounded sequence in $E$ has a weakly $p$-Cauchy subsequence. If $p=\infty$, we shall say simply that $E$ has the $wCSP$. }
\end{definition}

\begin{remark} \label{rem:wCSPp} {\em
(i) By the Rosenthal $\ell_1$ theorem, a Banach space $E$ has the $wCSP$ if and only if it has no an isomorphic copy of $\ell_1$.

(ii) It is known (see Corollary 7.3.8 of \cite{Talagrand}) that a Banach space $E$ has no an isomorphic copy of $\ell_1$ if and only if the dual Banach space $E'_\beta$ has the weak Radon--Nikodym property, and hence  if and only if $E$ has the $wCSP$.

(iii) If $1\leq p<q\leq\infty$ and an lcs $E$ has the $wCSP_p$, then $E$ has the $wCSP_q$.

(iv) If an lcs $E$ has the $wCSP$, then $E$ has the Rosenthal property.

(v) A  Banach space $E$ has the $wCSP_p$ if and only if the identity map $\Id_E:E\to E$ is weakly sequentially $p$-precompact.\qed}
\end{remark}

\begin{proposition} \label{p:Lp-wCSPp}
If $1<p<\infty$, then $\ell_p$ has the $wCSP_p$ if and only if $p\geq 2$.
\end{proposition}

\begin{proof}
Assume that $1< p<2$. Then $p<p^\ast$, and hence, by Proposition 
6.5 of \cite{Gab-Pel}, $\ell_p$ has the $p$-Schur property. Let $S=\{e_n\}_{n\in\w}$ be the canonical unit basis of $\ell_p$. Assuming that  $S$ has a weakly $p$-Cauchy subsequence $\{e_{n_k}\}_{k\in\w}$, the $p$-Schur property implies that it is Cauchy in $\ell_p$ which is impossible. Thus $\ell_p$ does not have the $wCSP_p$.

Assume that $2\leq p<\infty$.   Then, by Proposition 1.4 of \cite{CS} (or by Corollary 
13.11 of \cite{Gab-Pel}), $B_{\ell_p}$ is weakly sequentially $p^\ast$-compact. Since $1<p^\ast\leq 2$, we have $p^\ast\leq p$. Therefore any weakly $p^\ast$-convergent sequence is also weakly $p$-convergent. Thus $B_{\ell_p}$ is weakly sequentially $p$-compact, and hence $\ell_p$ has the $wCSP_p$.\qed
\end{proof}

\begin{proposition} \label{p:sum-wCSP}
Let $p\in[1,\infty]$, and let $\{E_i\}_{i\in I}$ be a non-empty family of locally convex spaces.
\begin{enumerate}
\item[{\rm(i)}] If $I=\w$, then $E=\prod_{i\in\w} E_i$ has the $wCSP_p$ if and only if each factor $E_i$ has the $wCSP_p$.
\item[{\rm(ii)}] The space $E=\bigoplus_{i\in I} E_i$ has the $wCSP_p$ if and only if each summand $E_i$ has the $wCSP_p$.
\end{enumerate}
\end{proposition}

\begin{proof}
(i) Assume that $E$ has the $wCSP_p$. Fix $j\in \w$, and let $S=\{x_{n,j}\}_{n\in\w}$ be a bounded sequence in $E_j$. If $\pi_j$ is the natural embedding of $E_j$ into $E$, then the sequence $\pi_j(S)$ is bounded in $E$ and hence, by the $wCSP_p$, it has a weakly $p$-Cauchy subsequence $\{\pi_j(x_{n_k,j})\}_{k\in\w}$. It is clear that $\{x_{n_k,j}\}_{k\in\w}$ is  weakly $p$-Cauchy in $E_j$. Thus $E_j$  has the $wCSP_p$.

Conversely, assume that all spaces $E_i$ have the $wCSP_p$, and let $\{x_n=(x_{n,i})_{i\in\w}\}_{n\in\w}$ be a sequence in $E$. We proceed by induction on $i\in\w$. For $i=0$,  the $wCSP_p$ of $E_0$ implies that there is a sequence $I_0$ in $\w$ such that the sequence $\{x_{j,0}\}_{j\in I_0}$ is  weakly $p$-Cauchy in $E_0$. For $i=1$,  the $wCSP_p$ of $E_1$ implies that there is a subsequence $I_1$ of $I_0$ such that the sequence $\{x_{j,1}\}_{j\in I_1}$ is  weakly $p$-Cauchy in $E_1$. Continuing this process we find a sequence $\w\supseteq I_0\supseteq I_1\supseteq\cdots$ such that the sequence $\{x_{j,i}\}_{j\in I_i}$ is  weakly $p$-Cauchy in $E_i$ for every $i\in\w$. For every $i\in\w$, choose $m_i\in I_i$ such that $m_i<m_{i+1}$ for all $i\in\w$. We claim that the subsequence $\{x_{m_i}\}_{i\in\w}$ of $\{x_n\}_{n\in\w}$ is weakly $p$-Cauchy in $E$. Indeed, since $E'=\bigoplus_{i\in\w} E'_i$, each $\chi\in E'$ has a form $\chi=(\chi_0,\cdots,\chi_k,0,\dots)$. Then for every strictly increasing sequence $(i_n)\subseteq \w$, we have
\begin{equation} \label{equ:wCSPp-1}
\big|\langle \chi, x_{m_{i_n}}-x_{m_{i_{n+1}}}\rangle\big|\leq \sum_{t=0}^k \big|\langle \chi_t, x_{m_{i_n},t}-x_{m_{i_{n+1}},t}\rangle\big|
\end{equation}
By the choice of $I_t$ and $(m_i)\in\w$, each sequence $\big\{  x_{m_{i_n},t}-x_{m_{i_{n+1}},t}\big\}_{n\in\w}$ is weakly $p$-summable. This and (\ref{equ:wCSPp-1}) imply that the sequence $\big\{  x_{m_{i_n}}-x_{m_{i_{n+1}}}\big\}_{n\in\w}$ is weakly $p$-summable, and hence  $\{x_{m_i}\}_{i\in\w}$  is weakly $p$-Cauchy in $E$. Thus $E$ has the $wCSP_p$.

(ii) The necessity can be proved repeating word for word the necessity in (i). The sufficiency follows from (i) because any bounded sequence in $E$ has finite support.\qed
\end{proof}

\begin{remark} {\em
The countability of the index set in (i) of Proposition \ref{p:sum-wCSP} is essential. Indeed, by the example in Lemma 
2.2 of \cite{Gab-Pel}, the product $\IR^\mathfrak{c}$ has a bounded sequence without weak Cauchy subsequences.\qed}
\end{remark}


\section{Permanent properties of Gelfand--Phillips types properties} \label{sec:perm-GP-property}



In this section we study relationships between different Gelfand--Phillips type  properties, stability under taking direct products and direct sums, and a connection of these properties with the Schur property. We also show that the class of locally convex spaces with the Gelfand--Phillips property is sufficiently large.


Recall that a locally convex space  $E$ is {\em von Neumann complete} if every precompact subset of $E$ is relatively compact. Following general Definition 11.12 of \cite{Gab-Pel}, 
$E$ is {\em $\mathsf{L}_{(p,q)}$-von Neumann complete} (resp., {\em $\mathsf{CL}_{(p,q)}$-von Neumann complete}) if every closed precompact set in $\mathsf{L}_{(p,q)}$ (resp., in $\mathsf{CL}_{(p,q)}$) is compact. Recall also that a locally convex space $E$ is called {\em semi-Montel} if every bounded subset of $E$ is relatively compact, and $E$ is a {\em Montel space} if it is a barrelled semi-Montel space.

\begin{lemma} \label{l:GPp=EGPp}
Let  $1\leq p\leq q\leq\infty$, and let $(E,\tau)$ be a  locally convex space.
\begin{enumerate}
\item[{\rm(i)}] If $E$ has the $EGP_{(p,q)}$ $($resp., $sEGP_{(p,q)}$$)$ property, then $E$ has the $GP_{(p,q)}$ $($resp., $sGP_{(p,q)}$$)$ property. The converse is true for $p$-barrelled spaces.
\item[{\rm(ii)}] If $1\leq p'\leq p$ and $q\leq q'\leq \infty$ and $E$ has the $GP_{(p',q')}$ $($resp., $sGP_{(p',q')}$, $EGP_{(p',q')}$, $sEGP_{(p',q')}$, $prGP_{(p',q')}$ or $spGP_{(p',q')}$$)$ property, then $E$ has  the $GP_{(p,q)}$ $($resp., $sGP_{(p,q)}$, $EGP_{(p,q)}$, $sEGP_{(p,q)}$, $prGP_{(p,q)}$ or $spGP_{(p,q)}$$)$ property.
\item[{\rm(iii)}]  If the class of relatively compact sets in $E$ coincides with class of relatively sequentially compact sets, then $E$ has the property $GP_{(p,q)}$ $($resp., $EGP_{(p,q)}$$)$ if and only if it has the property $sGP_{(p,q)}$ $($resp., $sEGP_{(p,q)}$$)$.
\item[{\rm(iv)}]  If $E$ is an angelic $p$-barrelled  space, then all the properties $GP_{(p,q)}$, $EGP_{(p,q)}$, $sGP_{(p,q)}$ and $sEGP_{(p,q)}$ are coincide for $E$.
\item[{\rm(v)}]  If $E$ has the $GP_{(p,q)}$ property then it has the $prGP_{(p,q)}$ property. The converse is true if $E$ is $\mathsf{L}_{(p,q)}$-von Neumann complete. 
\item[{\rm(vi)}] If $E$ is sequentially complete, then $E$ has the $spGP_{(p,q)}$  {\rm(}resp., $spEGP_{(p,q)}${\rm)} property if and only if it has the $sGP_{(p,q)}$  {\rm(}resp., $sEGP_{(p,q)}${\rm)} property.
\item[{\rm(vii)}] If $E$ is semi-Montel, then $E$ has the $EGP_{(p,q)}$ property.
\end{enumerate}
\end{lemma}

\begin{proof}
(i) follows from the easy fact that every $(p,q)$-limited set is $(p,q)$-$\EE$-limited, see Lemma 
3.1(i) of \cite{Gab-limited}, and from the definition of $p$-barrelled spaces.

(ii) follows from the fact that every $(p,q)$-($\EE$-)limited set is $(p',q')$-($\EE$-)limited, see Lemma 
3.1(vi) of \cite{Gab-limited}.

(iii), (v) and (vii) follow from the corresponding definitions.

(iv) follows from (i) and the fact that for angelic spaces, relatively compact sets are exactly relatively sequentially compact sets.
\smallskip

(vi) It suffices to prove the necessity. Assume that $E$ has the $spGP_{(p,q)}$ (resp., $spEGP_{(p,q)}$) property. Let $A$ be a $(p,q)$-limited  (resp., $(p,q)$-$\EE$-limited) set in $E$. Then $A$ is sequentially precompact. We show that $A$ is even relatively sequentially compact. Indeed, let $S=\{x_n\}_{n\in\w}$ be a sequence in $A$. Since $A$ is sequentially precompact, $S$ has a Cauchy subsequence $\{x_{n_k}\}_{k\in\w}$. Since $E$ is sequentially complete, there is $x\in E$ such that $x_{n_k}\to x$. Therefore $A$ is relatively sequentially compact. Thus $E$ has the $sGP_{(p,q)}$ (resp., $sEGP_{(p,q)}$) property.\qed
\end{proof}

If the space $E$ carries its weak topology (for example, $E=C_p(X)$), we have the following result.
\begin{lemma} \label{l:GPp=EGPp-2}
Let  $1\leq p\leq q\leq\infty$, and let $(E,\tau)$ be a  locally convex space such that $E=E_w$. Then:
\begin{enumerate}
\item[{\rm(i)}]  $E$ has the  $prEGP_{(p,q)}$ property and hence the  $prGP_{(p,q)}$ property;
\item[{\rm(ii)}] every bounded subset of $E$ is $(p,q)$-$\EE$-limited; consequently, if $E$ is von Neumann complete, then $E$ has the property  $EGP_{(p,q)}$.
\end{enumerate}
\end{lemma}

\begin{proof}
(i) is trivial because every bounded subset of $E$ is precompact.

(ii) Let $A$ be a  bounded subset of $E$, and let $S=\{\chi_n\}_{n\in\w}$ be an equicontinuous weak$^\ast$ $p$-summable sequence in $E'$. Since $E$ carries its weak topology it follows that $S\subseteq F^{\circ\circ}$ for some finite subset $F=\{\eta_1,\dots,\eta_k\}\subseteq E'$. Therefore $S\subseteq \spn(F)$, and hence  for every $n\in\w$,  there  are scalars $c_{1,n},\dots,c_{k,n}\in\IF$ such that
\[
\chi_n =c_{1,n} \eta_1 +\cdots +  c_{k,n} \eta_k.
\]
Since $S$ is   weak$^\ast$ $p$-summable, we obtain that $(c_{i,n})_n\in \ell_p$ (or $\in c_0$ if $p=\infty$) for every $i=1,\dots,k$. As $A$ is bounded, for every $i=1,\dots,k$, one can define $C_i:= \sup_{a\in A}  |\langle\eta_i,a\rangle|$. Then
\[
\sup_{a\in A} |\langle\chi_n,a\rangle|\leq \sum_{i=1}^k \sup_{a\in A} |c_{i,n} | |\langle\eta_i,a\rangle|= \sum_{i=1}^k  C_i |c_{i,n} |.
\]
Therefore $\big(\sup_{a\in A} |\langle\chi_n,a\rangle|\big)_n \in\ell_p\subseteq \ell_q$  (or $\in c_0$ if $p=q=\infty$).  Thus $A$ is a $(p,q)$-$\EE$-limited set.\qed
\end{proof}
Note that, by Theorem 
3.13 of \cite{Gab-Pel}, $C_p(X)$ is von Neumann complete if and only if $X$ is discrete.

\begin{lemma}  \label{l:tau-TTT-GP}
Let  $1\leq p\leq q\leq\infty$, and let $\tau$ be a locally convex compatible topology on a locally convex space $(E,\TTT)$ such that $\tau\subseteq\TTT$. If $(E,\TTT)$ has the $prGP_{(p,q)}$ {\rm(}resp., $spGP_{(p,q)}$, $GP_{(p,q)}$ or $sGP_{(p,q)}${\rm)} property, then also $(E,\tau)$ has the same property.
\end{lemma}

\begin{proof}
Since $\tau$ and $\TTT$ are compatible, the spaces  $(E,\tau)$  and $(E,\TTT)$ have the same $(p,q)$-limited subsets by Lemma 
3.1(vii) of \cite{Gab-limited}.  As by assumption all $(p,q)$-limited sets in  $(E,\TTT)$ are precompact (resp., sequentially precompact, relatively compact or relatively sequentially compact), the inclusion $\tau\subseteq\TTT$ implies that so are all $(p,q)$-limited sets also in $(E,\tau)$.\qed
\end{proof}

\begin{remark} \label{rem:tau-TTT-GP} {\em
The converse in Lemma \ref{l:tau-TTT-GP} is not true in general. Indeed, let $E=\ell_\infty$. Then $E_w$ trivially has the precompact $GP$ property. However, the space $E$ is not a Gelfand--Phillips space by \cite{Phillips}. Consequently, the property of being a Gelfand--Phillips space is not a property of the dual pair $(E,E')$. \qed}
\end{remark}

Below we consider  relationships between $b$-version of the Gelfand--Phillips properties and the Gelfand--Phillips properties.
\begin{lemma} \label{l:GPp=EGPp-3}
Let  $1\leq p\leq q\leq\infty$, and let $(E,\tau)$ be a  locally convex space. Then:
\begin{enumerate}
\item[{\rm(i)}]  if $E$ has the $b$-$GP_{(p,q)}$ property, then it has the $prGP_{(p,q)}$ property;
\item[{\rm(ii)}] if $E$ is barrelled, then $E$ has the $b$-$GP_{(p,q)}$ $($resp., $b$-$EGP_{(p,q)}$$)$ property if and only if it has the $prGP_{(p,q)}$ $($resp., $prEGP_{(p,q)}$$)$ property;
\item[{\rm(iii)}] if $E$ is barrelled and von Neumann complete, then $E$ has the $b$-$GP_{(p,q)}$ $($resp., $b$-$EGP_{(p,q)}$$)$ property if and only if it has the $GP_{(p,q)}$ $($resp., $EGP_{(p,q)}$$)$ property.
\end{enumerate}
\end{lemma}

\begin{proof}
(i) follows from the inclusion $\tau\subseteq \beta(E,E')$.

(ii) and (iii) follow from the equality $E=E_\beta$ and the corresponding definitions.\qed
\end{proof}

We select the next proposition which shows that for wide classes of locally convex spaces important for applications in fact there is only a unique Gelfand--Phillips type property (however, in general, these notions are different, see Examples \ref{exa:product-sGP} and \ref{exa:subspace-non-GP} and Theorem \ref{t:Cp-EGPp} below).

\begin{proposition} \label{p:equal-GP-strict-LF}
Let  $1\leq p\leq q\leq\infty$, and let $E$ be an angelic, complete  and barrelled space  {\rm(}for example, $E$ is a strict $(LF)$-space{\rm)}. Then all the properties  $GP_{(p,q)}$,  $EGP_{(p,q)}$, $sGP_{(p,q)}$, $sEGP_{(p,q)}$, $prGP_{(p,q)}$, $prEGP_{(p,q)}$, $spGP_{(p,q)}$, $spEGP_{(p,q)}$ and $b$-$GP_{(p,q)}$ are equivalent.
\end{proposition}

\begin{proof}
The properties  $GP_{(p,q)}$,  $EGP_{(p,q)}$, $sGP_{(p,q)}$ and  $sEGP_{(p,q)}$ are equivalent by (iv) of Lemma \ref{l:GPp=EGPp}.
The properties $GP_{(p,q)}$ and  $prGP_{(p,q)}$ are equivalent by (v) of Lemma \ref{l:GPp=EGPp}.
The properties $sGP_{(p,q)}$ and  $spGP_{(p,q)}$ are equivalent by (vi) of Lemma \ref{l:GPp=EGPp}.
The properties $prGP_{(p,q)}$ and  $b$-$GP_{(p,q)}$ are equivalent by (ii) of Lemma \ref{l:GPp=EGPp-3}.

Since $E$ is barrelled, the $(p,q)$-limited sets and the $(p,q)$-$\EE$-limited sets are the same. Therefore, the properties $prGP_{(p,q)}$ and  $prEGP_{(p,q)}$, as well as  the properties $spGP_{(p,q)}$ and  $spEGP_{(p,q)}$   are coincide. \qed
\end{proof}

By Proposition 
4.2 of \cite{Gab-limited}, every $p$-limited subset of $E$ is  coarse $p$-limited.
Although every precompact subset of a $c_0$-barrelled space is limited by Corollary 
3.7 of \cite{Gab-limited}, non-$c_0$-barrelled spaces may contain even convergent sequences which are not limited, see Example 
3.10 in \cite{Gab-limited}. We complement these results in the the following assertion.
Recall that a subset $A$ of a topological space $X$ is {\em functionally bounded} if $f(A)$ is bounded for every $f\in C(X)$; clearly, any compact subset of $X$ is functionally bounded.
\begin{proposition} \label{p:limited-coarse-p-limited}
Let $p\in[1,\infty]$, and let $E$ be a locally convex space. Then:
\begin{enumerate}
\item[{\rm(i)}] each functionally bounded subset  $A$ of $E$ is coarse $p$-limited;
\item[{\rm(ii)}] every limited subset of  $E$ is coarse $p$-limited.
\end{enumerate}
\end{proposition}

\begin{proof}
(i) Let $T:E\to\ell_p$ (or $T:E\to c_0$ if $p=\infty$) be an operator. Then $T(A)$ is a  functionally bounded subset of the metric space $\ell_p$ (or $c_0$), and hence $T(A)$ is relatively compact. Thus $A$ is coarse $p$-limited.

(ii) Let $A$ be a limited subset of $E$, and let $T\in\LL(E,\ell_p)$ (or $T\in\LL(E,c_0)$ if $p=\infty$). Then, by (iv) of Lemma 
3.1 of \cite{Gab-limited}, $T(A)$ is a limited subset of $\ell_p$ (or $c_0$). Since the Banach spaces $\ell_p$ and $c_0$ are separable, they have the Gelfand--Phillips property by Theorem \ref{t:Banach-GP}. Therefore $T(A)$ is relatively compact in $\ell_p$ (or in $c_0$). Thus $A$  is a coarse $p$-limited set in $E$.\qed
\end{proof}

Below we consider the coarse versions of the Gelfand--Phillips property.
\begin{proposition} \label{p:pq-lim-coarse-p-lim}
Let $1\leq p\leq q\leq\infty$, and let $(E,\tau)$ be a  locally convex space.
\begin{enumerate}
\item[{\rm(i)}] If $E$ has the coarse $GP_{p}$ $($resp., coarse $sGP_{p}$ or coarse $spGP_{p}$$)$ property, then $E$ has the coarse $prGP_{p}$ $($resp., coarse $spGP_{p}$ or coarse $prGP_{p}$$)$ property.
\item[{\rm(ii)}] If $E=E_w$, then $\mathsf{CL}_p(E)=\Bo(E)$ and hence $E$ has  the coarse $prGP_{p}$ property.
\item[{\rm(iii)}] If  $E$ has the coarse $GP_p$ {\rm(}resp., coarse $sGP_p$, coarse $prGP_p$ or coarse $spGP_p${\rm)} property, then $E$ has the $GP_{p}$ {\rm(}resp., $sGP_{p}$, $prGP_{p}$ or $spGP_{p}${\rm)} property.
\item[{\rm(iv)}] If  $E$ has the coarse $GP_p$ property, then $E$ has the $GP$ property. The converse is not true in general.
\item[{\rm(v)}] If $E$ is barrelled, then $E$ has the $GP_{(1,\infty)}$  {\rm(}resp., $sGP_{(1,\infty)}$, $prGP_{(1,\infty)}$, or $spGP_{(1,\infty)}${\rm)} property if and only if it has the coarse $GP_{1}$  {\rm(}resp.,  coarse  $sGP_{1}$,  coarse  $prGP_{1}$, or coarse  $spGP_{1}${\rm)} property.
\end{enumerate}
\end{proposition}

\begin{proof}
(i) is obvious.

(ii) Since $E=E_w$, Lemma 
5.12 of \cite{Gab-limited} implies that each operator from $E$ into any Banach space is finite-dimensional. Therefore any bounded subset of $E$ is coarse $p$-limited. As any bounded subset of $E$ is precomapct it follows that $E$ has  the coarse $prGP_{p}$ property.

(iii) follows from Proposition 4.2 
of \cite{Gab-limited} which states that every $p$-limited set is coarse $p$-limited.

(iv) follows from Proposition \ref{p:limited-coarse-p-limited}. To show that the converse is not true in general, let $E=c_0$ and $1\leq p<\infty$. Since $E$ is separable it has the $GP$ property by Theorem \ref{t:Banach-GP}. However, $E$ does not have the coarse $GP_p$ property because, by the Pitt theorem, $B_E$ is a non-compact coarse $p$-limited set. 

(v) follows from Proposition 
3.13 of \cite{Gab-limited} which states that for a barrelled space $E$, the class of $(1,\infty)$-limited sets coincides with the class of coarse $1$-limited sets.\qed
\end{proof}

It is natural to consider relationships between $V^\ast_{(p,q)}$ type properties of Pe{\l}czy\'{n}ski and Gelfand--Phillips type properties. Following \cite{Gab-DP}, an lcs $E$ has the ({\em weak}) {\em  Glicksberg property} if $E$ and $E_w$ have the same (resp., absolutely convex) compact  sets. Note that the weak Glicksberg property does not imply the Schur property and vise versa.
\begin{proposition} \label{p:V*pq=>GPpq}
Let  $1\leq p\leq q\leq\infty$, and let $(E,\tau)$ be a  locally convex space.
\begin{enumerate}
\item[{\rm(i)}] Let $E$ have the weak  Glicksberg property.  If $E$ has the property $V^\ast_{(p,q)}$ $($resp.,  $EV^\ast_{(p,q)}$$)$, then $E$ has the $GP_{(p,q)}$ $($resp., $EGP_{(p,q)}$$)$ property.
\item[{\rm(ii)}] Let $E$ have the Schur property.  If $E$ has the property $sV^\ast_{(p,q)}$ $($resp., $sEV^\ast_{(p,q)}$$)$, then $E$ has the $sGP_{(p,q)}$ $($resp., $sEGP_{(p,q)}$$)$ property.
\end{enumerate}
\end{proposition}

\begin{proof}
(i) Recall that every $(p,q)$-limited  (resp., $(p,q)$-$\EE$-limited) set $A$ of $E$ is a also a $(p,q)$-$(V^\ast)$ (resp.,  $(p,q)$-$(EV^\ast)$) set by Lemma 
3.1(viii) of \cite{Gab-limited}. By (iii) of Lemma 7.2 of \cite{Gab-Pel},  
$A^{\circ\circ}$ is also a $(p,q)$-$(V^\ast)$ (resp.,  $(p,q)$-$(EV^\ast)$) set.  By the property $V^\ast_{(p,q)}$ (resp.,  $EV^\ast_{(p,q)}$), $A^{\circ\circ}$ is weakly compact, and hence, by  the weak Glicksberg property, $A^{\circ\circ}$ is a compact subset of $E$. Therefore $A$ is relatively compact. Thus $E$ has the $GP_{(p,q)}$ (resp., $EGP_{(p,q)}$) property.
\smallskip

(ii) Let $A$ be a $(p,q)$-limited  (resp., $(p,q)$-$\EE$-limited) set in $E$. Then, by Lemma 
3.1(viii) of \cite{Gab-limited}, $A$ is  a $(p,q)$-$(V^\ast)$ (resp.,  $(p,q)$-$(EV^\ast)$) set. By the property $sV^\ast_{(p,q)}$ (resp.,  $sEV^\ast_{(p,q)}$), $A$ is relatively weakly sequentially compact.  Since $E$ is a Schur space, Lemma 
2.1 of \cite{Gab-limited} implies that $A$ is relatively sequentially compact in $E$. Thus $E$ has the $sGP_{(p,q)}$ (resp., $sEGP_{(p,q)}$) property.\qed
\end{proof}

\begin{remark}{\em
In both cases (i) and (ii) of Proposition \ref{p:V*pq=>GPpq} the converse is not true in general. Indeed, let $E=c_0$ and $p=q=\infty$. Being separable $c_0$ has the $GP$ property and hence, by Proposition \ref{p:equal-GP-strict-LF}, the $sGP$ property. On the other hand, $E$ has neither the property $V^\ast_\infty$ nor the property $sV^\ast_\infty$ because the Schur property of $E'_\beta=\ell_1$ implies that $B_E$ is a non-compact $\infty$-$(V^\ast)$ set.\qed}
\end{remark}

Below we show that the classes of locally convex spaces with Gelfand--Phillips type properties are stable under taking direct products and direct sums.
\begin{proposition} \label{p:product-sum-GP}
Let $1\leq p\leq q\leq\infty$, and let $\{E_i\}_{i\in I}$  be a non-empty family of locally convex spaces. Then:
\begin{enumerate}
\item[{\rm(i)}] $E=\prod_{i\in I} E_i$ has the $GP_{(p,q)}$ $($resp., $EGP_{(p,q)}$, $prGP_{(p,q)}$, $prEGP_{(p,q)}$, coarse $GP_{p}$, or coarse $prGP_{p}$$)$ property if and only if all spaces $E_i$ have the same property;
\item[{\rm(ii)}] $E=\bigoplus_{i\in I} E_i$ has the $GP_{(p,q)}$  $($resp., $EGP_{(p,q)}$, $sGP_{(p,q)}$, $sEGP_{(p,q)}$, $prGP_{(p,q)}$, $prEGP_{(p,q)}$, $spGP_{(p,q)}$,  $spEGP_{(p,q)}$, coarse $GP_{p}$, coarse $sGP_{p}$, coarse $spGP_{p}$, or coarse $prGP_{p}$$)$ property if and only if all spaces  $E_i$ have  the same property;
\item[{\rm(iii)}] if $I=\w$ is countable, then $E=\prod_{i\in \w} E_i$ has the  $sGP_{(p,q)}$ $($resp.,  $sEGP_{(p,q)}$, $spGP_{(p,q)}$, $spEGP_{(p,q)}$, coarse $sGP_{p}$, or coarse $spGP_{p}$$)$ property  if and only if all spaces $E_i$ have the same property.
\end{enumerate}
\end{proposition}

\begin{proof}
We consider only $GP_{(p,q)}$ type properties because the coarse $GP_{p}$ type properties can be considered analogously using properties of coarse $p$-limited sets, see Lemma 
4.1 and Proposition 
4.3 of \cite{Gab-limited}.

To prove the necessity, let $E$ have the  $GP_{(p,q)}$ (resp.,  $EGP_{(p,q)}$, $sGP_{(p,q)}$, $sEGP_{(p,q)}$, $prGP_{(p,q)}$, $prEGP_{(p,q)}$, $spGP_{(p,q)}$  or $spEGP_{(p,q)}$) property. Fix $j\in I$, and let $A_j$ be a $(p,q)$-limited (resp., $(p,q)$-$\EE$-limited)  set in $E_j$. Since $E_j$ is a direct summand of  $E$, to show that $A_j$ is relatively  compact (resp., precompact, relatively sequentially compact or sequentially precompact) in $E_j$ it suffices to show that the set $A:=A_j\times \prod_{i\in I\SM \{j\}} \{0_i\}$ has the same property in $E$. In all cases (i)--(iii), by Proposition 
3.3 of \cite{Gab-limited}, $A$ is a $(p,q)$-limited (resp., $(p,q)$-$\EE$-limited) set in $E$ (also as the image of $A_j$ under the canonical embedding of $E_j$ into $E$). Since $E$ has the $GP_{(p,q)}$ (resp.,  $EGP_{(p,q)}$, $sGP_{(p,q)}$, $sEGP_{(p,q)}$, $prGP_{(p,q)}$, $prEGP_{(p,q)}$, $spGP_{(p,q)}$  or $spEGP_{(p,q)}$)  property, it follows that $A$ is relatively   compact (resp., precompact, relatively sequentially compact or sequentially precompact) in $E$, as desired.

To prove the sufficiency, assume that all spaces $E_i$ have the  $GP_{(p,q)}$ (resp.,  $EGP_{(p,q)}$, $sGP_{(p,q)}$, $sEGP_{(p,q)}$, $prGP_{(p,q)}$, $prEGP_{(p,q)}$, $spGP_{(p,q)}$  or $spEGP_{(p,q)}$) property, and let  $A$ be a $(p,q)$-limited (resp., $(p,q)$-$\EE$-limited)  set in $E$.
Recall that Lemma 2.1 of \cite{Gab-Pel} 
states that a subset of a countable product of Tychonoff spaces is (relatively) sequentially compact if and only if so are all its projections. This result
and, for the cases (i) and (iii), the fact that the product of relatively compact sets (resp., precompact sets) is relatively compact (resp., precompact), and for the case (ii), the fact that the support of $A$ is finite since $A$ is bounded,  to show that $A$ is relatively   compact (resp., precompact, relatively sequentially compact or sequentially precompact) it suffices to prove that for every $i\in I$, the projection $A_i$ of $A$ onto the $i$th coordinate has the same property in $E_i$. But this condition is satisfied because,  by (iv) of Lemma 3.1 
of \cite{Gab-limited}, the projection $A_i$ is a $(p,q)$-limited (resp., $(p,q)$-$\EE$-limited)  set in $E_i$ and hence it is relatively   compact (resp., precompact, relatively sequentially compact or sequentially precompact) in $E_i$ by the corresponding property of $E_i$.\qed
\end{proof}

In (iii) of Proposition \ref{p:product-sum-GP} the countability of the set $I$ is essential.

\begin{example} \label{exa:product-sGP}
Let $1\leq p\leq q\leq\infty$. Then the separable space $\IR^\mathfrak{c}$ has the $EGP_{(p,q)}$ property and hence the $GP_{(p,q)}$ property and the coarse $prGP_{p}$ property, but it has neither the  $spGP_{(p,q)}$ property nor the coarse $spGP_{p}$ property.
\end{example}

\begin{proof}
The space $\IR^\mathfrak{c}$ has the $EGP_{(p,q)}$ property and the coarse $prGP_{p}$ property by (i) of Proposition \ref{p:product-sum-GP}. By (i) of Proposition 
3.3 of \cite{Gab-limited}, any bounded subset $A$ of $\IR^\mathfrak{c}$ is ${(p,q)}$-limited (and hence also ${(p,q)}$-$\EE$-limited) and, by (i) of Proposition 
 4.3 of \cite{Gab-limited}, $A$ is also coarse $p$-limited. In particular, identifying $3^\w$ with $\mathfrak{c}$ we obtain that the sequence $S=\{f_n\}_{n\in\w}$ of functions constructed in Lemma 2.2 of \cite{Gab-Pel} 
is a ${(p,q)}$-limited set. However, it is proved in Lemma 2.2 of \cite{Gab-Pel} 
that $S$ is not sequentially precompact. Thus $\IR^\mathfrak{c}$ has neither the $spGP_{(p,q)}$ property  nor the coarse $spGP_{p}$ property.\qed
\end{proof}



In the next proposition we consider the heredity of Gelfand--Phillips type properties. Recall that a subspace $Y$ of a Tychonoff space $X$ is {\em sequentially closed} if $Y$ contains all limits of convergent sequences from $Y$.

\begin{proposition} \label{p:subspace-pGP}
Let $1\leq p\leq q\leq\infty$, and let $L$ be a subspace of a locally convex space $E$.
\begin{enumerate}
\item[{\rm(i)}] If $L$ is closed in $E$ and $E$ has  the  $GP_{(p,q)}$ $($resp., $EGP_{(p,q)}$ or coarse $GP_p$$)$ property, then also $L$ has the same property.
\item[{\rm(ii)}] If $L$ is sequentially closed in $E$ and $E$ has  the $sGP_{(p,q)}$ $($resp., $sEGP_{(p,q)}$  or coarse $sGP_p$$)$ property, then also $L$ has the same property.
\item[{\rm(iii)}]  If $E$ has  the $prGP_{(p,q)}$ $($resp., $prEGP_{(p,q)}$ or coarse $prGP_p$$)$ property,  then also $L$ has the same property.
\item[{\rm(iv)}]  If $E$ has  the $spGP_{(p,q)}$ $($resp., $spEGP_{(p,q)}$ or coarse $spGP_p$$)$ property,  then also $L$ has the same property.
\end{enumerate}
\end{proposition}

\begin{proof}
Below we consider only $GP_{(p,q)}$ type properties because the coarse $GP_{p}$ type properties can be considered analogously using properties of coarse $p$-limited sets, see Lemma 
4.1 of \cite{Gab-limited}.

Let $A$ be a $(p,q)$-limited  (resp., $(p,q)$-$\EE$-limited) set in $L$. Then, by  (iv) of Lemma 
3.1 of \cite{Gab-limited} applied to the identity embedding $\Id_L:L\to E$, $A$ is  a $(p,q)$-limited  (resp., $(p,q)$-$\EE$-limited) set also in $E$.

(i) By the property $GP_{(p,q)}$ (resp., $EGP_{(p,q)}$) of $E$, the set $A$ is relatively compact in $E$. Since $L$ is a closed subspace of $E$ it follows that $A$ is relatively compact in $L$. Thus $L$ has the property  $GP_{(p,q)}$ (resp., $EGP_{(p,q)}$).

(ii) By the property  $sGP_{(p,q)}$ (resp., $sEGP_{(p,q)}$) of $E$, the set  $A$ is relatively  sequentially compact in $E$.  Since $L$ is a sequentially closed subspace of $E$ it follows that $A$ is relatively sequentially compact in $L$. Thus $L$ has  the property $sGP_{(p,q)}$ (resp., $sEGP_{(p,q)}$).

(iii)  By the property $prGP_{(p,q)}$ $($resp., $prEGP_{(p,q)}$), the set  $A$ is precompact in $E$ and hence also in $L$. Thus $L$ has the property $prGP_{(p,q)}$ $($resp., $prEGP_{(p,q)}$).

(iv) By the property $spGP_{(p,q)}$ (resp., $spEGP_{(p,q)}$) of $E$, the set  $A$ is sequentially precompact in $E$. Let $\{a_n\}_{n\in\w}$ be a sequence in $A$. Take a subsequence $\{a_{n_k}\}_{k\in\w}$ of $\{a_n\}_{n\in\w}$ which is Cauchy in $E$. Since $L$ is a subspace of $E$ it follows that $\{a_{n_k}\}_{k\in\w}$  is Cauchy also in $L$. Therefore $A$ is sequentially precompact in $L$. Thus $L$ has  the property $spGP_{(p,q)}$ (resp., $spEGP_{(p,q)}$).\qed
\end{proof}

Below we characterize spaces $C_p(X)$ which have one of the equicontinuous $GP$ properties.
\begin{theorem} \label{t:Cp-EGPp}
Let $1\leq p\leq q\leq\infty$, and let $X$ be a Tychonoff space. Then:
\begin{enumerate}
\item[{\rm(i)}]  $C_p(X)$ has the  $EGP_{(p,q)}$ property if and only if $X$ is discrete;
\item[{\rm(ii)}] $C_p(X)$ has the $b$-$EGP_{(p,q)}$ property if and only if every functionally bounded subset of $X$ is finite;
\item[{\rm(iii)}] $C_p(X)$ has the  $prEGP_{(p,q)}$ property and hence the  $prGP_{(p,q)}$ property.
\end{enumerate}
\end{theorem}

\begin{proof}
(i) Assume that $C_p(X)$ has  the $EGP_{(p,q)}$ property. Then, by (ii) of Lemma \ref{l:GPp=EGPp-2},  the bounded set $B=\{f\in C_p(X): \|f\|_X \leq 1\}$ is $(p,q)$-$\EE$-limited. Therefore, by  the $EGP_{(p,q)}$ property, $B$ has compact closure in $C_p(X)$. As $B$ is dense in the compact space $\mathbb{D}^X$ we obtain that $\mathbb{D}^X\subseteq C_p(X)$. But this is possible only if $X$ is discrete.

Conversely, assume that $X$ is discrete. Then clearly every bounded subset of $C_p(X)=\IF^X$ has compact closure. Thus $C_p(X)$ has the $EGP_{(p,q)}$ property.

(ii) Assume that $C_p(X)$ has the $b$-$EGP_{(p,q)}$ property, and let $A$ be a functionally bounded subset of $X$. Assuming that $A$ is infinite we can find a sequence $\{a_n\}\subseteq A$ and a sequence $\{U_n\}_{n\in\w}$ of open subsets of $X$ such that $a_n\in U_n$ and $U_n\cap U_m=\emptyset$ for all distinct $n,m\in\w$. For every $n\in\w$, choose $f_n\in C_p(X)$ such that $f_n(X\SM U_n)=\{0\}$ and $f_n(a_n)=n$. Then the sequence $S=\{f_n\}_{n\in\w}$ is a bounded subset of $C_p(X)$ and hence, by (ii) of Lemma \ref{l:GPp=EGPp-2}, $S$ is a ${(p,q)}$-$\EE$-limited set. Therefore to get a contradiction with the $b$-$EGP_{(p,q)}$ property of $C_p(X)$ it suffices to show that $S$ is not barrel-precompact. To this end, we note that the functional boundedness of $A$ implies that the set  $B:=\{f\in C_p(X): \|f\|_A\leq 1\}$ is a barrel in $C_p(X)$. Since, by construction,  $S\nsubseteq k B$ for every $k\in\w$ we obtain that $S$ is not barrel-precompact.

Conversely, assume that every functionally bounded subset of $X$ is finite. We show that every $(p,q)$-$\EE$-limited subset $B$ of $C_p(X)$ is barrel-precompact. To this end, we note that, by the Buchwalter--Schmetz theorem,  the space $C_p(X)$ is barrelled. Now, let $D$ be a barrel in $C_p(X)$. Then $D$ is a neighborhood of zero. Since $C_p(X)$ carries its weak topology and $B$ is its bounded subset, it follows that $B$ is precompact  and hence there is a finite subset $F\subseteq C_p(X)$ such that $B\subseteq F+D$. Thus $B$ is barrel-precompact.

(iii) Since $C_p(X)$ carries its weak topology, the assertion follows from (i) of Lemma \ref{l:GPp=EGPp-2}.\qed
\end{proof}

In (i) and (ii) of Proposition \ref{p:subspace-pGP} the condition on $L$ of being a (sequentially) closed subspace of $E$ is essential even for Fr\'{e}chet spaces as the next example shows.

\begin{example} \label{exa:subspace-non-GP}
Let $1\leq p\leq q\leq\infty$, and let $X$ be a countable non-discrete Tychonoff space whose compact subsets are finite. Then the dense subspace $C_p(X)$ of $\IF^\w$ has no any of the properties $GP_{(p,q)}$,  $EGP_{(p,q)}$, $sGP_{(p,q)}$ or $sEGP_{(p,q)}$, but it has the properties $prEGP_{(p,q)}$, $prGP_{(p,q)}$, $spEGP_{(p,q)}$ and $spGP_{(p,q)}$.
\end{example}

\begin{proof}
By the Buchwalter--Schmetz theorem, the space $C_p(X)$ is barrelled and hence it is $p$-barrelled. Being metrizable $C_p(X)$ is an angelic space. Thus,  by (iv)  of Lemma \ref{l:GPp=EGPp} and (i)  of Theorem \ref{t:Cp-EGPp}, $C_p(X)$ has no any of the properties $GP_{(p,q)}$,  $EGP_{(p,q)}$, $sGP_{(p,q)}$ or $sEGP_{(p,q)}$.

On the other hand,  by (iii) of Theorem \ref{t:Cp-EGPp} and the metrizability of $\IF^\w$, the space $C_p(X)$ has the properties $prEGP_{(p,q)}$, $prGP_{(p,q)}$, $spEGP_{(p,q)}$ and $spGP_{(p,q)}$. \qed 
\end{proof}


Our next purpose is to show that any locally convex space is a quotient space of a locally convex space with the (sequential) $GP_{(p,q)}$ property for all $1\leq p\leq q\leq\infty$. First we recall some definitions.

The {\em  free locally convex space}  $L(X)$ over a Tychonoff space $X$ is a pair consisting of a locally convex space $L(X)$ and  a continuous map $i: X\to L(X)$  such that every  continuous map $f$ from $X$ to a locally convex space  $E$ gives rise to a unique continuous linear operator $\Psi_E(f): L(X) \to E$  with $f=\Psi_E(f) \circ i$. The free locally convex space $L(X)$ always exists and is essentially unique.
For  $\chi = a_1 x_1+\cdots +a_n x_n\in L(X)$ with distinct $x_1,\dots, x_n\in X$ and  nonzero $a_1,\dots,a_n\in\IF$, we set $\chi[x_i]:=a_i$ and
\[
\| \chi\|:=|a_1|+\cdots +|a_n|, \; \mbox{ and } \; \mathrm{supp}(\chi):=\{ x_1,\dots, x_n\}.
\]
From the definition of $L(X)$ it easily follows the well known fact that the dual space $L(X)'$ of $L(X)$ is linearly isomorphic to the space $C(X)$ with the pairing
\[
\langle f,\chi\rangle=  a_1 f(x_1)+\cdots +a_n f(x_n)\quad \mbox{ for $f\in C(X)$}.
\]
For every subset $A\subseteq L(X)$, let $\mathrm{supp}(A):=\bigcup_{\chi\in A} \mathrm{supp}(\chi)$ and $C_A:=\sup\big(\{ \| \chi\|: \chi\in A\}\cup \{0\}\big)$.

\begin{theorem} \label{t:L(X)-GPpq}
Let $1\leq p\leq q\leq\infty$, and let $L(X)$ be the free locally convex space over a Tychonoff space $X$. Then:
\begin{enumerate}
\item[{\rm(i)}] each $(p,q)$-limited subset of $L(X)$ is finite-dimensional;
\item[{\rm(ii)}] $L(X)$ has the $GP_{(p,q)}$ property, the sequential $GP_{(p,q)}$ property and  the $b$-$GP_{(p,q)}$ property;
\item[{\rm(iii)}] every locally convex space $E$ is a quotient space of $L(E)$.
\end{enumerate}
\end{theorem}

\begin{proof}
(i) Suppose for a contradiction that there is a one-to-one infinite-dimensional sequence $A=\{\chi_n\}_{n\in\w}$ in $L(X)\SM\{0\}$ which is a $(p,q)$-limited set. Since the support of any $\chi\in L(X)$ is finite and $A$ is infinite-dimensional, without loss of generality we can assume that $A$ satisfies the following condition:
\begin{enumerate}
\item[{\rm(a)}] for every $n\in\w$, there is $x_{n+1}\in \supp(\chi_{n+1})\SM \bigcup_{i\leq n} \supp(\chi_i)$.
\end{enumerate}
Fix an arbitrary $x_0\in\supp(\chi_0)$. Passing to a subsequence of $A$ if needed, we can assume also that there is a sequence $\{U_n\}_{n\in\w}$ of open subsets of $X$ such that
\begin{enumerate}
\item[{\rm(b)}]  $U_{n}\cap \supp(\chi_n)=\{x_{n}\}$ for every $n\in\w$;
\item[{\rm(c)}]  $U_n\cap U_m=\emptyset$ for all distinct $n,m\in\w$.
\end{enumerate}
For every $n\in\w$ and taking into account (b), choose $f_n\in C(X)$ such that
\begin{enumerate}
\item[{\rm(d)}]  $f_n\big(X\SM U_n\big) =\{0\}$;
\item[{\rm(e)}]  $\langle f_n,\chi_n\rangle= f_n(x_n)\cdot \chi_n[x_n]=1$.
\end{enumerate}
Since the support of any $\chi\in L(X)$ is finite, (c) and (d) imply that the sequence  $\{f_n\}_{n\in\w}$ is weak$^\ast$ $p$-summable in $C(X)=L(X)'$ for every $p\in[1,\infty]$. On the other hand, (e) implies
\[
\sup_{\chi\in A} |\langle f_n,\chi\rangle| \geq \sup_{i\in\w}|\langle f_n,\chi_i\rangle|\geq \langle f_n,\chi_n\rangle =1 \not\to 0,
\]
which means that $A$ is not a $(p,q)$-limited set, a contradiction.
\smallskip

(ii) immediately follows from (i).

(iii) By the definition of $L(E)$, the identity map $\Id_E:E\to E$ can be extended to an operator $I:L(E)\to E$. Since $E$ is a closed subspace of $L(E)$, it is clear that $I$ is a quotient map.\qed
\end{proof}

For the coarse $p$-limited subsets of $L(X)$ the situation is antipodal to the $(p,q)$-limited subsets of $L(X)$.
\begin{theorem} \label{t:free-coarse-p}
Let $p\in[1,\infty]$, and let $L(X)$ be the free locally convex space over a Tychonoff space $X$. Then:
\begin{enumerate}
\item[{\rm(i)}] each bounded subset $B$ of $L(X)$ is coarse $p$-limited;
\item[{\rm(ii)}] $L(X)$ has the coarse $GP_{p}$ property if and only if each functionally bounded subset of $X$ is finite.
\end{enumerate}
\end{theorem}

\begin{proof}
(i) Recall that  Proposition  
2.7 of \cite{Gabr-free-lcs} states that a subset $A$ of $L(X)$ is bounded if and only if $\supp(A)$ is functionally bounded in $X$ and $C_A$ is finite.  Let $B$ be a bounded subset of $L(X)$ and let $T\in\LL\big(L(X),\ell_p\big)$ (or $T\in\LL\big(L(X),c_0\big)$ if $p=\infty$). Then $\supp(B)$ is functionally bounded in $X$, and hence $T(\supp(B))$ is functionally bounded in the Banach space $\ell_p$ (or in $c_0$). Since any metric space is a $\mu$-space, the closure $K$ of $T(\supp(B))$ is compact. Now the inclusion $T(B) \subseteq C_B\cdot \cacx(K)$ implies that $T(B)$ is a relatively compact subset of $\ell_p$ (or of $c_0$). Thus $B$ is a coarse $p$-limited subset of $L(X)$.
\smallskip

(ii) Assume that $L(X)$ has the coarse $GP_{p}$ property. Then, by (i), every bounded subset $B$ of $L(X)$ has compact closure. In particular, $L(X)$ is quasi-complete. Therefore, by Theorem 
3.8 of \cite{Gab-limited},  each functionally bounded subset of $X$ is finite. Conversely, if  all functionally bounded subsets of $X$ are finite, then, by Proposition  
2.7 of \cite{Gabr-free-lcs}, any bounded subset of $L(X)$ is finite-dimensional and hence relatively compact. Thus $L(X)$ trivially has  the coarse $GP_{p}$ property.\qed
\end{proof}

Theorems \ref{t:L(X)-GPpq} and \ref{t:free-coarse-p} suggest the following problem.
\begin{problem}
Let $1\leq p\leq q\leq\infty$. Characterize Tychonoff spaces $X$ for which $L(X)$ has one of the $V^\ast_{(p,q)}$ type properties defined in Definition \ref{def:property-Vp*}.
\end{problem}

We know from (ii) of Lemma \ref{l:GPp=EGPp} that the case $(1,\infty)$ is the strongest one in the sense that if $E$ has some $GP_{(1,\infty)}$ type property, then it has the same $GP_{(p,q)}$ type property for all $1\leq p\leq q\leq\infty$. This result and (v)  of Proposition \ref{p:pq-lim-coarse-p-lim} motivate the problem of whether, for example, the $spGP_{(1,\infty)}$ property implies some addition  conditions on the space $E$. We answer this problem in the affirmative in the rest of this section.

\begin{theorem} \label{t:LF-coarse-GP1}
Let $E$ be a strict $(LF)$-space. If $E$ does not contain an isomorphic copy of $\ell_1$ which is complemented in $E$, then the following assertions are equivalent:
\begin{enumerate}
\item[{\rm(i)}] $E$ has the $GP_{(1,\infty)}$ property;
\item[{\rm(ii)}] $E$ has the coarse $GP_{1}$ property;
\item[{\rm(iii)}] $E$ is a Montel space.
\end{enumerate}
If {\rm(i)--(iii)} hold, then $E$ is separable.
\end{theorem}

\begin{proof}
Since  strict $(LF)$-spaces are barrelled  by Proposition 11.3.1 of \cite{Jar}, (i) and (ii) are equivalent by (v) of Proposition \ref{p:pq-lim-coarse-p-lim}. Corollary 
5.16 of \cite{Gab-limited} states that a subset of $E$ is bounded if and only if it is a coarse $1$-limited set. Therefore $E$ has the coarse $GP_{1}$ property if and only if every bounded subset of $E$ is relatively compact, i.e. $E$ is a Montel space.

To prove that $E$ is separable, let $E:=\SI E_n$, where all $E_n$ are closed Fr\'{e}chet subspaces of $E$. By Proposition 11.5.4(b) of \cite{Jar} and (iii), all spaces $E_n$ are Montel. Therefore, by Theorem 11.6.2 of \cite{Jar}, all spaces $E_n$ are separable. Thus also $E=\bigcup_{n\in\w} E_n$ is separable.\qed
\end{proof}

\begin{remark} {\em
Let $\Gamma$ be an uncountable set and let $1<p<\infty$. Then,  by Theorem \ref{t:LF-coarse-GP1}, the reflexive non-separable Banach space $\ell_p(\Gamma)$ has neither the $GP_{(1,\infty)}$ property nor the coarse $GP_{1}$ property. However, $\ell_p(\Gamma)$ has the coarse $GP_{p}$ property by Remark 3(2) of \cite{GalMir},  and it has the Gelfand--Phillips property by Proposition \ref{p:pq-lim-coarse-p-lim}.\qed}
\end{remark}

\begin{proposition} \label{p:wsGP-Schur}
Let $E$ be a sequentially complete barrelled locally convex space. If $E$ has the $spGP_{(1,\infty)}$ property, then $E$ has the Schur property.
\end{proposition}

\begin{proof}
Let $S=\{x_n\}_{n\in\w}$ be a weakly null sequence in $E$. Then $S$ is weakly sequentially precompact. Therefore, by Corollary 
3.14 of \cite{Gab-limited}, $S$ is a $(1,\infty)$-limited set. Hence, by the $spGP_{(1,\infty)}$ property, $S$ is sequentially precompact. 
Thus, by Lemma \ref{l:null-seq}, $x_n\to 0$ in $E$ and hence $E$ has the Schur property.\qed
\end{proof}

In general, the converse in Proposition \ref{p:wsGP-Schur} is not true as Example \ref{exa:product-sGP} shows. Nevertheless, in some important cases the converse holds true.

\begin{proposition} \label{p:wsGP-Schur-con}
Let $E$ be a sequentially complete barrelled locally convex space whose bounded subsets are weakly sequentially precompact. Then $E$ has the $spGP_{(1,\infty)}$ property if and only if it has the Schur property.
\end{proposition}

\begin{proof}
The necessity follows from Proposition \ref{p:wsGP-Schur}.  To prove the sufficiency, assume that $E$ has the Schur property and let $A$ be a $(1,\infty)$-limited subset of $E$. Since $A$ is bounded it is weakly sequentially precompact. Then, by Lemma 
2.2 of \cite{Gab-limited}, $A$ is sequentially precompact. Thus $E$ has the $spGP_{(1,\infty)}$ property. \qed
\end{proof}

\begin{corollary} \label{c:wsGP-Schur-con}
Let $E$ be a strict $(LF)$-space that does not contain an isomorphic copy of $\ell_1$. Then $E$ has the Schur property if and only if it has one {\rm(}and hence all{\rm)} of the properties  $GP_{(1,\infty)}$,  $EGP_{(1,\infty)}$, $sGP_{(1,\infty)}$, $sEGP_{(1,\infty)}$, $prGP_{(1,\infty)}$, $spGP_{(1,\infty)}$, $prEGP_{(1,\infty)}$, $spEGP_{(1,\infty)}$ and $b$-$GP_{(1,\infty)}$.
\end{corollary}

\begin{proof}
By Proposition \ref{p:equal-GP-strict-LF}, it suffices to consider only the $spGP_{(1,\infty)}$ property. It is well known that $E$ is a complete barrelled space. By the Ruess Theorem 
\cite{ruess} mentioned before Theorem \ref{t:inf-V*-wsc}, $E$ has the Rosenthal property. Since $E$ has no an isomorphic copy of $\ell_1$, it follows that every bounded sequence has a weakly Cauchy subsequence (i.e. every bounded subset is weakly sequentially precompact). Now Proposition \ref{p:wsGP-Schur-con} applies. \qed
\end{proof}


\section{Characterizations  of Gelfand--Phillips types properties} \label{sec:char-GP-property}


In this section we characterize locally convex spaces with Gelfand--Phillips types properties, in particular, in some important partial cases.

Let $p,q\in[1,\infty]$, and let $E$ and $L$ be locally convex spaces.   Generalizing the notions of limited, limited completely continuous and limited $p$-convergent operators between Banach spaces introduced in \cite{BourDies}, \cite{SaMo} and \cite{FZ-p}, respectively, and following \cite{Gab-p-Oper}, a linear map $T:E\to L$ is called  {\em $(p,q)$-limited} if $T(U)$ is a $(p,q)$-limited subset of $L$ for some $U\in\Nn_0(E)$; if $p=q$ or $p=q=\infty$ we shall say that $T$ is {\em $p$-limited} or {\em limited}, respectively.

\begin{theorem} \label{t:GPpq-property-oper}
Let $1\leq p\leq q\leq \infty$, and let $E$ be a locally convex space. Then  the following assertions are equivalent:
\begin{enumerate}
\item[{\rm(i)}] $E$ has the $prGP_{(p,q)}$ {\rm(}resp., $GP_{(p,q)}$, $sGP_{(p,q)}$ or $spGP_{(p,q)}${\rm)} property;
\item[{\rm(ii)}] for every locally convex space $L$, if an operator $T:L\to E$ transforms bounded sets of $L$ to $(p,q)$-limited  sets of $E$, then $T$ transforms bounded sets of $L$ to precompact  {\rm(}resp., relatively compact, relatively sequentially compact or sequentially precompact{\rm)} subsets of $E$;
\item[{\rm (iii)}] every $(p,q)$-limited operator $T:L\to E$ from a locally convex  space $L$ to $E$ is precompact  {\rm(}resp., compact, sequentially compact or sequentially precompact{\rm)};
\item[{\rm (iv)}] every $(p,q)$-limited operator $T:L\to E$ from a normed space $L$ to $E$ is precompact   {\rm(}resp., compact, sequentially compact or sequentially precompact{\rm)}.
\end{enumerate}
If $E$ is sequentially complete, then {\rm(i)--(iv)} are equivalent to the following
\begin{enumerate}
\item[{\rm(v)}] every $(p,q)$-limited operator $T:L\to E$ from a Banach space $L$ to $E$ is precompact   {\rm(}resp., compact, sequentially compact or sequentially precompact{\rm)}.
\end{enumerate}
\end{theorem}

\begin{proof}
(i)$\Rightarrow$(ii) Let $T:L\to E$ be an operator which transforms bounded sets of $L$ to $(p,q)$-limited  sets of $E$, and let $B$ be a bounded subset of $L$. By assumption $T(B)$ is a  $(p,q)$-limited subset of $E$. Then, by (i), $T(B)$ is a precompact (resp., relatively compact, relatively sequentially compact or sequentially precompact) subset of $E$, as desired.
\smallskip

The implications (ii)$\Ra$(iii)$\Ra$(iv) are trivial.
\smallskip

(iv)$\Rightarrow$(i) Fix a $(p,q)$-limited  subset $B$ of $E$. It is clear that the closed absolutely convex hull $D$ of the set $B$ in $E$ is also $(p,q)$-limited. Let $L$ be the linear hull of $D$. Since $D$ is bounded in $E$, the function $\|\cdot\|:L\to[0,\infty)$, $\|\cdot\|:x\mapsto\inf\{r\ge0:x\in rD\}$ is a well-defined norm on the linear space $L$ and the set $D$ coincides with the closed unit ball $B_L$ of the normed space $(L,\|\cdot\|)$. Since the  identity inclusion $T:(L,\|\cdot\|)\to E$ is continuous and the set $D=T(B_L)$ is $(p,q)$-limited in $E$, the operator $T$ is $(p,q)$-limited. By (iv), the set $D=T(B_L)$ is precompact  (resp., relatively compact, relatively sequentially compact or sequentially precompact) in $E$, and hence so is the set $B\subseteq D$. Thus $E$ has the $prGP_{(p,q)}$ (resp., $GP_{(p,q)}$, $sGP_{(p,q)}$ or $spGP_{(p,q)}$) property.
\smallskip

Assume that $E$ is sequentially complete. Then the implication (iv)$\Ra$(v) is trivial. To prove  the implication (v)$\Ra$(iv), let $T:L\to E$ be a $(p,q)$-limited operator from a normed space $L$ to $E$. Then, by Proposition 
3.7 of \cite{Gab-Pel}, $T$ can be extended to a bounded operator ${\bar T}$ from the completion $\overline{L}$ of $L$ to $E$. Observe that ${\bar T}(B_{\overline{L}}) \subseteq \overline{T(B_L)}$ and hence ${\bar T}$ is also $(p,q)$-limited. Thus, by (v), the operator ${\bar T}$ and hence also $T$ are precompact   (resp., compact, sequentially compact or sequentially precompact). \qed
\end{proof}

Let $1\leq p\leq q\leq \infty$, and let $E$ and $L$ be locally convex spaces.
Following Definition 
16.1  of \cite{Gab-Pel}, a linear map $T:E\to L$ is called {\em $(q,p)$-convergent} if it sends weakly $p$-summable sequences in $E$ to strongly $q$-summable sequences in $L$. Recall (see \S~19.4 of \cite{Jar}) that a sequence $\{x_n\}_{n\in\w}$ in an lcs $L$ is called {\em strongly $p$-summable} if $\big(q_U(x_n)\big)\in\ell_p$ (or $\big(q_U(x_n)\big)\in c_0$ if $p=\infty$) for every $U\in\Nn_0^c(E)$, where as usual $q_U$ denotes the gauge functional of $U$.


\begin{theorem} \label{t:sequential-GPpq-property}
Let $1\leq p\leq q\leq \infty$, and let  $E$ be a locally convex space. Then the following assertions are equivalent:
\begin{enumerate}
\item[{\rm(i)}] if $L$ is a  locally convex space  and $T:L\to E$ is an operator such that $T^\ast: E'_{w^\ast} \to L'_\beta$ is $(q,p)$-convergent, then $T$ transforms bounded sets of $L$ to relatively sequentially compact  {\rm(}resp., sequentially precompact{\rm)} subsets of $E$;
\item[{\rm(ii)}] if $L$ is a normed space and $T:L\to E$ is an operator such that $T^\ast: E'_{w^\ast} \to L'_\beta$ is $(q,p)$-convergent, then $T$ is sequentially compact {\rm(}resp., sequentially precompact{\rm)};
\item[{\rm(iii)}] the same as {\rm(ii)} with $L=\ell_1^0$;
\item[{\rm(iv)}] $E$ has the $sGP_{(p,q)}$ property  {\rm(}resp., the $spGP_{(p,q)}$ property{\rm)}.
\end{enumerate}
Moreover, if $E$ is locally complete, then {\rm(i)-(iv)} are equivalent to
\begin{enumerate}
\item[{\rm(v)}] the same as {\rm(ii)} with $L=\ell_1$.
\end{enumerate}
\end{theorem}

\begin{proof}
The implications (i)$\Ra$(ii)$\Ra$(iii) and (ii)$\Ra$(v) are trivial.

(iii)$\Ra$(iv) and (v)$\Ra$(iv): Let $A$ be a $(p,q)$-limited subset of $E$. Fix an arbitrary sequence $S=\{x_n\}_{n\in\w}$ in $A$, so $S$ is a bounded subset of $E$. Therefore, by Proposition 
5.9 of \cite{Gab-limited}, the linear map $T:\ell_1^0 \to E$ (or $T:\ell_1 \to E$ if $E$ is locally complete) defined by
\[
T(a_0 e_0+\cdots+a_ne_n):=a_0 x_0+\cdots+ a_n x_n \quad (n\in\w, \; a_0,\dots,a_n\in\IF).
\]
is continuous. For every $n\in\w$ and each $\chi\in E'$, we have
$
\langle T^\ast(\chi),e_n\rangle=\langle \chi,T(e_n)\rangle=\langle\chi,x_n\rangle
$
and hence $T^\ast(\chi)=\big(\langle\chi,x_n\rangle\big)_n\in\ell_\infty$. In particular, $\|T^\ast(\chi)\|_{\ell_\infty}=\sup_{n\in\w} |\langle\chi,x_n\rangle|$.

Let now $\{\chi_n\}_{n\in\w}$ be a weak$^\ast$ $p$-summable sequence in $E'_{w^\ast}$. Since $A$ and hence also $S$ are $(p,q)$-limited sets we obtain $\big(\|T^\ast(\chi_n)\|_{\ell_\infty}\big)=\big(\sup_{i\in\w} |\langle\chi_n,x_i\rangle|\big)\in\ell_q$ (or $\in c_0$ if $q=\infty$). Therefore $T^\ast$ is $(q,p)$-convergent, and hence, by (iii) or (v), the operator $T$ is sequentially compact (resp., sequentially precompact). Therefore $S=\{T(e_n)\}_{n\in\w}$ has a  convergent (resp.,  Cauchy) subsequence in $E$. Hence $A$ is a relatively sequentially compact (resp., sequentially precompact) subset of $E$. Thus $E$ has the $sGP_{(p,q)}$ property  (resp., the $spGP_{(p,q)}$ property).
\smallskip

(iv)$\Ra$(i)  Let $T:L\to E$ be  an operator from a locally convex space $(L,\tau)$ such that $T^\ast: E'_{w^\ast} \to L'_\beta$ is $(q,p)$-convergent. Fix an arbitrary bounded subset $B$ of $L$. Then $U:=B^\circ$ is a neighbourhood of zero of $L'_\beta$. Let $\{\chi_n\}_{n\in\w}$ be a weakly $p$-summable sequence in $E'_{w^\ast}$. By assumption $T^\ast: E'_{w^\ast} \to L'_\beta$ is $(q,p)$-convergent, and hence
\begin{equation} \label{equ:sequential-GPpq-property}
\big( q_U(T^\ast(\chi_n))\big)\in\ell_q \quad \mbox{(or $\in c_0$ if $q=\infty$)},
\end{equation}
where $q_U$ is the gauge functional of $U$. For every $\chi\in E'$, we have
\[
\begin{aligned}
q_U(T^\ast(\chi)) & =\inf\big\{\lambda>0 : T^\ast(\chi)\in \lambda B^\circ\big\}=\inf\big\{\lambda>0 : \sup_{x\in B} |\langle T^\ast(\chi),x\rangle|\leq \lambda\big\}\\
& =\sup_{x\in B} |\langle T^\ast(\chi),x\rangle|=\sup_{x\in B} |\langle \chi,T(x)\rangle|.
\end{aligned}
\]
Therefore, (\ref{equ:sequential-GPpq-property}) means that the set $T(B)$ is a $(p,q)$-limited subset of $E$. By (iv), we obtain that $T(B)$ is a relatively sequentially compact  (resp., sequentially precompact) subsets of $E$. \qed
\end{proof}


Let $p,q,q'\in[1,\infty]$, $q'\leq q$, and let $E$ and $L$ be locally convex spaces. Following \cite{Gab-p-Oper}, a linear map $T:E\to L$ is called {\em $(q',q)$-limited $p$-convergent} if $T(x_n)\to 0$ for every  weakly $p$-summable sequence $\{x_n\}_{n\in\w}$ in $E$ which is a $(q',q)$-limited subset of $E$. 

The next our purpose is to give another characterization of locally convex spaces $E$ with the sequentially precompact $GP_{(q',q)}$ property under an additional assumption that $E$ has the $wCSP$. First we prove the following lemma.

\begin{lemma} \label{l:wsGP-id}
Let $1\leq p\leq q\leq \infty$, and let $E$ be a locally convex space. If  $E$ has  the $spGP_{(p,q)}$ property, then the identity map $\Id_E: E\to E$ is $(p,q)$-limited $\infty$-convergent.
\end{lemma}

\begin{proof}
Fix a $(p,q)$-limited weakly null sequence $S=\{x_n\}_{n\in\w}$  in $E$, and suppose for a contradiction that $x_n\not\to 0$ in $E$. Passing to a subsequence if needed we can assume that $S\cap U=\emptyset$ for some neighborhood $U$ of zero in $E$. Since $S$ is a $(p,q)$-limited set and $E$ has the $spGP_{(p,q)}$ property, $S$ is sequentially precompact hence precompact in $E$. As $S$ is also weakly null, Lemma \ref{l:null-seq}, 
implies that $x_n\to 0$. But this is impossible since $S\cap U=\emptyset$.\qed
\end{proof}




\begin{theorem} \label{t:sGPp-limited-p-conv}
Let $1\leq p\leq q\leq\infty$, and let $E$ be a locally convex space with the $wCSP$. Then $E$ has the $spGP_{(p,q)}$ property if and only if the identity map $\Id_E: E\to E$ is $(p,q)$-limited $\infty$-convergent.
\end{theorem}

\begin{proof}
The necessity immediately follows from Lemma \ref{l:wsGP-id}. To prove the sufficiency assume that the identity map $\Id_E: E\to E$ is $(p,q)$-limited $\infty$-convergent, and let $A$ be a $(p,q)$-limited  subset of $E$. We have to show that any sequence $S=\{x_n\}_{n\in\w}$  in $A$ contains a Cauchy subsequence (in $E$).  Since $E$ has the  $wCSP$, $S$ has a subsequence $S'=\{x_{n_k}\}_{k\in\w}$ which is weakly Cauchy. Let $\{k_i\}_{i\in\w}$ be a strictly increasing sequence in $\w$. Then $\{x_{n_{k_i}}-x_{n_{k_{i+1}}}\}_{k\in\w}$ is weakly null. As $\{x_{n_{k_i}}-x_{n_{k_{i+1}}}\}_{k\in\w}$ is also a $(p,q)$-limited set and $\Id_E$ is $(p,q)$-limited $\infty$-convergent, it follows that $x_{n_{k_i}}-x_{n_{k_{i+1}}}\to 0$ in $E$. Thus $S'$ is Cauchy in $E$.\qed
\end{proof}

We know from Theorem \ref{t:Drew-GP} that a Banach space $E$ has the Gelfand--Phillips property if and only if every limited weakly null sequence in $E$ is norm null. An analogue assertion for the coarse $GP_p$ property (and $2\leq p<\infty$) was obtained in Theorem 3 of \cite{GalMir}. Below we generalize both these results.

\begin{theorem} \label{t:char-coarse-prGPp}
Let $1\leq p\leq q\leq\infty$, and let $(E,\tau)$ be a locally convex space.
\begin{enumerate}
\item[{\rm(i)}] If $E$ has the $prGP_{(p,q)}$ property {\rm(}the $spGP_{(p,q)}$ property{\rm)}, then every weakly null $(p,q)$-limited sequence in $E$ is $\tau$-null.
\item[{\rm(ii)}] If $E$ has the coarse $prGP_{p}$ property {\rm(}the coarse $spGP_{p}$ property{\rm)}, then every weakly null coarse $p$-limited sequence in $E$ is $\tau$-null.
\item[{\rm(iii)}] Assume additionally that $p\geq 2$ and $E$ has the Rosenthal property. Then $E$ has the $prGP_{(p,q)}$ property if and only if every weakly null $(p,q)$-limited sequence in $E$ is $\tau$-null.
\item[{\rm(iv)}] Assume additionally that $p\geq 2$ and $E$ has the Rosenthal property. Then $E$ has the coarse $prGP_{p}$ property if and only if every weakly null coarse $p$-limited  sequence in $E$ is $\tau$-null.
\end{enumerate}
\end{theorem}

\begin{proof}
(i) and (ii): Let $S=\{x_n\}_{n\in\w}$ be a weakly null $(p,q)$-limited (resp., coarse $p$-limited) sequence in $E$. By the $prGP_{(p,q)}$ property (resp., the $spGP_{(p,q)}$ property, the coarse $prGP_{p}$ property or the coarse $spGP_{p}$ property), $S$ is a precompact subset of $E$. Thus, by Lemma \ref{l:null-seq}, 
$S$ is $\tau$-null.

(iii) and (iv): Assume that $p\geq 2$ and $E$ has the Rosenthal property. The necessity is proved in (i) and (ii). To prove the sufficiency, let every weakly null $(p,q)$-limited (resp., coarse $p$-limited) sequence in $E$ is $\tau$-null. Let $A$ be a  $(p,q)$-limited (resp., coarse $p$-limited) subset of $E$, and suppose for a contradiction that $A$ is not $\tau$-precompact. Then there are a sequence $\{a_n\}_{n\in\w}$ in $A$ and $U\in\Nn_0(E)$ such that $a_n-a_m\not\in U$ for all distinct $n,m\in\w$. By Theorem 
5.21 of \cite{Gab-limited} and Theorem \ref{t:inf-V*-wsc}, we obtain that $A$ is weakly sequentially precompact. Therefore, passing to a subsequence if needed we can assume that $\{a_n\}_{n\in\w}$ is weakly Cauchy. In particular, the sequence $\{a_n-a_{n+1}\}_{n\in\w}$ is weakly null. Since $\{a_n-a_{n+1}\}_{n\in\w}$ is also a  $(p,q)$-limited (resp., coarse $p$-limited)  subset of $E$ (see Lemmas 3.1 and 4.1 of \cite{Gab-limited}), 
it follows that $a_n-a_{n+1} \to 0$ in $E$ which contradicts the condition $a_n-a_{n+1}\not\in U$ for all $n\in\w$. Thus $A$ is $\tau$-precompact and hence $E$ has the  $prGP_{(p,q)}$ property (resp., the coarse  $prGP_{p}$ property).\qed
\end{proof}

Let $D$ be a subset of a locally convex space $E$. Recall that a subset $A$ of $E$ is called {\em $D$-separated} if $a-a'\not\in D$ for all distinct $a,a'\in A$, and $A$ is {\em separated} if $A$ is $U$-separated for some neighborhood $U$ of zero.
We also denote by $(\ell_r)_p$ and $(c_0)_p$ the Banach spaces $\ell_r$ and $c_0$ endowed with the topology induced from $\IF^\w$.

Since any limited set and each $(p,q)$-limited set are $(p,\infty)$-limited, the case $q=\infty$ is of interest. Below we give necessary conditions to have the precompact $GP_{(p,\infty)}$ property.
\begin{proposition} \label{p:GPp-necessary}
Let $p\in[1,\infty]$, and let $E$ be a locally convex space. Consider the following assertions:
\begin{enumerate}
\item[{\rm (i)}] $E$ has the precompact $GP_{(p,\infty)}$ property;
\item[{\rm (ii)}] For every bounded non-precompact set $B\subseteq E$, there is a weak$^\ast$ $p$-summable sequence $\{\chi_n\}_{n\in\w}$ in $E'$ such that $\|\chi_n\|_B \not\to 0$.
\item[{\rm (iii)}] For any infinite bounded  separated subset $D$ of $E$ and every $\delta>0$ there exist a sequence $\{x_n\}_{n\in\w}$ in $D$ and a  weak$^\ast$ $p$-summable sequence $\{f_n\}_{n\in\w}$ in $E'$ such that $|f_n(x_k)|<\delta$ and $|f_n(x_n)|>1+\delta$ for all natural numbers $k<n$.
\item[{\rm (iv)}] For any infinite bounded separated set $D$ in $E$ there exists a continuous operator $T:E\to (\ell_p)_p$ {\rm(}or $T:E\to (c_0)_p$ if $p=\infty${\rm)} such that $T(D)$ is not precompact in the Banach space $\ell_p$ {\rm(}or in $c_0$ if $p=\infty${\rm)}.
\end{enumerate}
Then {\rm(i)$\Ra$(ii)$\Ra$(iii)$\Ra$(iv)}.
\end{proposition}

\begin{proof}
(i)$\Rightarrow$(ii) Fix any bounded non-precompact set $B\subseteq E$. By (i), $E$ has the  precompact  $GP_{(p,\infty)}$ property and hence $B$ is not $(p,\infty)$-limited. Therefore there exists a weak$^\ast$  $p$-summable sequence $\{\chi_n\}_{n\in\w}$ in $E'$ such that $\|\chi_n\|_B\not\to 0$, as desired.
\smallskip

(ii)$\Rightarrow$(iii) Fix any  $\delta>0$ and any infinite bounded  separated set $D$ in  $E$. Choose $U\in\Nn_0(E)$ such that the set $D$ is  $U$-separated. Observe that $D$ is not precompact because $D$ is $U$-separated and infinite. By (ii), there is a weak$^\ast$ $p$-summable sequence $\{g_n\}_{n\in\w}$ in $E'$ such that $\|g_n\|_D \not\to 0$. Passing to a subsequence if needed we can assume that $\|g_n\|_D >a>0$ for all $n\in\w$. Further, multiplying each functional $g_n$ by $(1+\delta)/a$ we can assume that
\begin{equation} \label{equ:GPp-1}
\|g_n\|_D = \sup_{d\in D}|g_n(d)| >1+\delta \quad \mbox{ for every }\; n\in\w.
\end{equation}
Now, choose an arbitrary $x_0\in D$ such that $|g_0(x_0)|>1+\delta$. Assume that, for $k\in\w$, we found $x_0,\dots,x_k\in D$ and a sequence $0=n_0 <n_1<\cdots<n_k$ of natural numbers such that
\[
|g_{n_j}(x_i)|<\delta \; \mbox{ and } \; |g_{n_j}(x_j)| >1+\delta \quad \mbox{ for every } \; 0\leq i<j\le k.
\]
Since $g_n\to 0$ in the weak$^\ast$ topology $\sigma(E',E)$, (\ref{equ:GPp-1}) implies that there are $n_{k+1}>n_k$ and $x_{k+1} \in D\SM \{x_0,\dots,x_k\}$ such that  $|g_{n_{k+1}}(x_{k+1})| >1+\delta$ and $|g_{n_{k+1}}(x_i)|<\delta$ for every $i\le k$. For every $k\in\w$, put $f_k:=g_{n_k}$, and observe that the subsequence $\{f_k\}_{k\in\w}$ of $\{g_n\}_{n\in\w}$ is also  weak$^\ast$ $p$-summable in $E'$ and
\[
|f_k(x_i)|=|g_{n_k}(x_i)|<\delta \; \mbox{ and }\; |f_k(x_k)|=|g_{n_k}(x_k)|>1+\delta
\]
for any numbers $i<k$, as desired.
\smallskip

(iii)$\Rightarrow$(iv)  Let $D$ be any bounded separated set in $E$. By (iii) applied for $\delta=\tfrac{1}{2}$, there exist a sequence $\{x_n\}_{n\in\w}$ in $D$ and a weak$^\ast$ $p$-summable sequence $\{\chi_n\}_{n\in\w}$ in $E'$ such that $|\chi_m(x_n)|<\frac{1}{2}$ and $|\chi_m(x_m)|>\frac{3}{2}$ for all $n<m$. Then $|\chi_m(x_m)-\chi_m(x_n)|\ge |\chi_m(x_m)|-|\chi_m(x_n)|>\frac{3}{2}-\frac{1}{2}=1$ for every $n<m$. It follows that the operator
\[
T: E\to  (\ell_p)_p \;\; (\mbox{or } T:E\to (c_0)_p \mbox{ if $p=\infty$}), \quad T(x):= \big(\chi_n(x)\big)_{n\in\w},
\]
is well-defined and continuous. Since
\[
\|T(x_n)-T(x_m)\|\ge |\chi_m(x_m)-\chi_m(x_n)|>1\;\; \mbox{ for every $n<m$},
\]
the set $T(D)\supseteq\{T(x_n)\}_{n\in\w}$ is not precompact in $\ell_p$ (or in $c_0$  if $p=\infty$).\qed
\end{proof}

In two extreme cases when $p=1$ or $p=\infty$ we can reverse implications in Proposition \ref{p:GPp-necessary}, cf. Theorem 2.2 of \cite{BG-GP-lcs}.

\begin{theorem} \label{t:GPpq-characterization}
Let $p\in\{1,\infty\}$. For a locally convex space $E$ the following assertions are equivalent:
\begin{enumerate}
\item[{\rm (i)}] $E$ has the precompact $GP_{(p,\infty)}$ property; 
\item[{\rm (ii)}] For every bounded non-precompact set $B\subseteq E$, there is a weak$^\ast$ $p$-summable sequence $\{\chi_n\}_{n\in\w}$ in $E'$ such that $\|\chi_n\|_B \not\to 0$.
\item[{\rm (iii)}] For any infinite bounded  separated subset $D$ of $E$ and every $\delta>0$ there exist a sequence $\{x_n\}_{n\in\w}$ in $D$ and a  weak$^\ast$ $p$-summable sequence $\{f_n\}_{n\in\w}$ in $E'$ such that $|f_n(x_k)|<\delta$ and $|f_n(x_n)|>1+\delta$ for all natural numbers $k<n$.
\item[{\rm (iv)}] For any infinite bounded separated set $D$ in $E$ there exists a continuous operator $T:E\to (\ell_1)_p$ {\rm(}or $ T:E\to (c_0)_p$ if $p=\infty${\rm)} such that $T(D)$ is not precompact in the Banach space $\ell_1$ {\rm(}or in $c_0$ if $p=\infty${\rm)}.
\end{enumerate}
\end{theorem}

\begin{proof}
The implications (i)$\Ra$(ii)$\Ra$(iii)$\Ra$(iv) are proved in Proposition \ref{p:GPp-necessary}.
\smallskip

(iv)$\Rightarrow$(i) To show that $E$ has the  precompact $GP_{(p,\infty)}$ property we show that any non-precompact subset of $E$ is not $(p,\infty)$-limited. Fix a bounded non-precompact subset $P\subseteq E$. Then there exists $U\in\Nn_0(E)$ such that $L\not\subseteq F+U$ for any finite subset $F\subseteq E$. For every $n\in\w$, choose inductively a point $z_n\in P$ so that $z_n\notin \bigcup_{k<n}(z_k+U)$. Observe that the  set  $D:=\{z_n:n\in\w\}$ is infinite, bounded and $U$-separated, so non-precompact. By (iv), there exists a continuous operator $T:E\to (\ell_1)_p$ (or $T:E\to (c_0)_p$ if $p=\infty$) such that the set $T(D)$ is not precompact in the Banach space $\ell_1$ (or in $c_0$ if $p=\infty$). We distinguish between two cases.

{\em Case 1. Assume that $p=\infty$.} We follow the idea of Theorem 2.2 of \cite{BG-GP-lcs}. Since $c_0$ is a Banach space and $T(D)$ is not precompact, there exist a sequence $\{ a_n\}_{n\in\w}$ in $D$ and $\delta>0$ such that  $\| T(a_n)-T(a_m)\|_{c_0}\geq \delta$ for all distinct $n,m\in\w$ (see also (ii) of Proposition \ref{p:compact-ell-p}).
Observe that the sequence $\{ T(a_n)\}_{n\in\w}$ is bounded in the Banach space $c_0$. Therefore there are two sequences $0\leq n_0<n_1<\cdots$ and $0\leq m_0< m_1<\cdots$ of natural numbers such that
\[
\big| \langle e'_{m_k},T(a_{n_k})\rangle\big| >\tfrac{\delta}{2} \quad \mbox{ for every }\; k\in\w,
\]
where $e'_n: (c_0)_p\to \IF$ is the $n$th coordinate functional. For every $k\in\w$, set  $f_k :=e'_{m_k} \circ T$. It follows that $\{f_k\}_{k\in\w}$ is a weak$^\ast$ null in $E'$ and
\[
\|f_k\|_P =\sup_{x\in P} |f_k(x)| \geq |f_k (a_{n_k})|>\tfrac{\delta}{2}
\]
for every $ k\in\w$, witnessing that the set $P$ is not limited.
\smallskip

{\em Case 2. Assume that $p=1$.} Recall that, by (i) of Proposition \ref{p:compact-ell-p},  a  bounded subset $A$ of $\ell_1$ is precompact if and only if
\begin{equation} \label{equ:GPp-2}
\lim_{m\to\infty} \sup\Big\{ \sum_{m\leq n} |x_n| : x=(x_n)\in A\Big\} =0.
\end{equation}
For every $k\in\w$, set $y_k:=T(z_k)=(a_{n,k})_{n\in\w}\in\ell_1$. Since $T(D)$ is non-precompact, (\ref{equ:GPp-2}) implies that there is $\e>0$ such that
\begin{equation} \label{equ:GPp-3}
\sup\Big\{ \sum_{m\leq n} |a_{n,k}| : y_k=(a_{n,k})\in T(D)\Big\} >10\e \quad \mbox{ for every $m\in\w$}.
\end{equation}
Let $i=0$. Set $m_0:=0$. By (\ref{equ:GPp-3}), choose $k_0,r_0\in\w$ such that $m_0 < r_0$ and
\begin{equation} \label{equ:GPp-4}
\sum_{n=m_0}^{r_0} |a_{n,k_0}| > 8\e.
\end{equation}
Choose $m_1>r_0$ such that
\begin{equation} \label{equ:GPp-5}
\sum_{n\geq m_1} |a_{n,k}| < \e \quad \mbox{ for every }\; k\leq k_0.
\end{equation}
For $i=1$, (\ref{equ:GPp-3})--(\ref{equ:GPp-5})  imply that there are $k_1,r_1\in\w$ such that $k_1>k_0$, $m_1 < r_1$ and
\begin{equation} \label{equ:GPp-6}
\sum_{n=m_1}^{r_1} |a_{n,k_1}| > 8\e.
\end{equation}
Choose $m_2>r_1$ such that
\[
\sum_{n\geq m_2} |a_{n,k}| < \e  \quad \mbox{ for every }\; k\leq k_1.
\]
Continuing this process we find  sequences $0=m_0<r_0<m_1<r_1<\cdots$ and $k_0<k_1<\cdots$ such that for every $i\in\w$, we have
\begin{equation} \label{equ:GPp-7}
\sum_{n=m_i}^{r_i} |a_{n,k_i}| > 8\e \;\; \mbox{ and }\;\; \sum_{n\geq m_{i+1}} |a_{n,k_i}| < \e.
\end{equation}

For every $i\in\w$, by Lemma 6.3 of \cite{Rudin},  there is a subset $F_i$ of $[m_i,r_i]$ such that
\begin{equation} \label{equ:GPp-8}
\Big|\sum_{n\in F_i} a_{n,k_i}\Big| > 2\e.
\end{equation}
For every $n\in\w$, let $e^\ast_n$ be the $n$th coordinate functional of the dual space $(\ell_{1})_p'$ of $(\ell_{1})_p$. For every $i\in\w$, set
\[
\chi_i:= \sum_{n\in F_i} e^\ast_n.
\]
Observe that the sequence  $\{\chi_i\}_{i\in\w}$ is weak$^\ast$ $1$-summable in $(\ell_{1})_p'$ because if $y=(a_n)\in\ell_1$, then the inequalities $m_i<r_i<m_{i+1}$ imply that all $F_i$ are pairwise disjoint and hence
\[
\sum_{i\in\w} |\langle \chi_i ,y\rangle|=\sum_{i\in\w} \Big|\sum_{n\in F_i} a_{n}\Big| \leq\sum_{i\in\w} \sum_{n\in F_i} |a_n| \leq \sum_{n\in\w} |a_n|=\|y\|_{\ell_1} <\infty.
\]
Therefore the sequence $\{T^\ast(\chi_i)\}_{i\in\w}$ is weak$^\ast$ $1$-summable in $E'$. For every $i\in\w$, the inequality  (\ref{equ:GPp-8}) implies
\[
\sup_{x\in P} |\langle T^\ast(\chi_i),x\rangle| \geq |\langle\chi_i,T(z_{k_i})\rangle|=|\langle\chi_i,y_{k_i}\rangle|= \Big|\sum_{n\in F_i} a_{n,k_i}\Big| > 2\e,
\]
which means that the set $P$ is not $(1,\infty)$-limited.\qed
\end{proof}

The next theorem immediately follows from Corollary 
3.18 of \cite{Gab-limited} and, under some conditions, it characterizes the (precompact) $GP_{(p,\infty)}$ property.
\begin{theorem} \label{t:wGP-Mackey-c0}
Let $E$ be a $p$-barrelled Mackey space such that $E'_{w^\ast}$ is a weakly $p$-angelic space. Then $E$ has the precompact $GP_{(p,\infty)}$ property. If in addition $E$ is von Neumann complete, then $E$ has the $GP_{(p,\infty)}$ property.
\end{theorem}

Since the case $p=\infty$ is of independent interest, we select the next corollary which follows from Corollary 
3.19 of \cite{Gab-limited} (and in fact from Theorem \ref{t:wGP-Mackey-c0}).

\begin{corollary} \label{c:precompact-GP}
Let a locally convex space $E$ satisfy one of the following conditions:
\begin{enumerate}
\item[{\rm(i)}] $E$ is a $c_0$-barrelled Mackey space such that $E'_{w^\ast}$ is a weakly angelic space;
\item[{\rm(ii)}] $E$ is a reflexive space  such that $E'_{\beta}$ is a weakly angelic space;
\item[{\rm(iii)}] $E$ is a separable $c_0$-barrelled Mackey space.
\end{enumerate}
Then $E$ has the precompact $GP$ property. If in addition $E$ is von Neumann complete, then $E$ has the $GP$ property.
\end{corollary}

Let $p\in[1,\infty]$, and let $E$ and $L$ be locally convex spaces. Following \cite{Gab-p-Oper}, an operator $T\in \LL(E,L)$ is called {\em  coarse $p$-limited}  if there is $U\in \Nn_0(E)$ such that $T(U)$ is a coarse $p$-limited  subset of $L$.

Banach spaces with the coarse $p$-Gelfand--Phillips property are characterized in Proposition 13 of \cite{GalMir}. The next theorem generalizes that result (we use in the proof the easy fact that a subset $A$ of $E$ is precompact if and only if each sequence in $A$ is precompact).

\begin{theorem} \label{t:coarse-GPp-char}
For $p\in[1,\infty]$ and a locally convex space $E$, the following assertions are equivalent:
\begin{enumerate}
\item[{\rm(i)}] $E$ has the coarse $sGP_p$ property  {\rm(}resp., the coarse $spGP_p$ property or the coarse $prGP_p$ property{\rm)};
\item[{\rm(ii)}] for every locally convex space $Y$, if an operator $T:Y\to E$ transforms bounded sets of $Y$ to coarse $p$-limited sets of $E$, then $T$ transforms bounded sets of $Y$ to relatively sequentially compact  {\rm(}resp., sequentially precompact or precompact{\rm)} subsets of $E$;
\item[{\rm(iii)}] for every normed space $Y$, each coarse $p$-limited operator $T:Y\to E$ is sequentially compact  {\rm(}resp., sequentially precompact or precompact{\rm)};
\item[{\rm(iv)}] each coarse $p$-limited operator $T:\ell_1^0\to E$ is sequentially compact  {\rm(}resp., sequentially precompact or precompact{\rm)}.
\end{enumerate}
If in addition $E$ is locally complete, then {\rm(i)-(iv)} are equivalent to
\begin{enumerate}
\item[{\rm(v)}] each coarse $p$-limited operator $T:\ell_1\to E$ is sequentially compact  {\rm(}resp., sequentially precompact or precompact{\rm)}.
\end{enumerate}
\end{theorem}

\begin{proof}
The implications (i)$\Ra$(ii)$\Ra$(iii)$\Ra$(iv) and (iii)$\Ra$(v) are trivial.

(iv)$\Ra$(i) and (v)$\Ra$(i): Let $A$ be a coarse $p$-limited subset of $E$. To show that $A$ is relatively sequentially compact (resp., sequentially precompact or precompact), let $\{x_n\}_{n\in\w}$ be a sequence in $A$. Then, by Proposition 
14.9 of \cite{Gab-Pel}, the linear map $T:\ell_1^0 \to E$ (or $T:\ell_1 \to E$ if $E$ is locally complete) defined by
\[
T(a_0 e_0+\cdots+a_ne_n):=a_0 x_0+\cdots+ a_n x_n \quad (n\in\w, \; a_0,\dots,a_n\in\IF).
\]
is continuous. It is clear that $T(B_{\ell_1^0})\subseteq \cacx\big(\{x_n\}_{n\in\w}\big)$ (or $T(B_{\ell_1})\subseteq \cacx\big(\{x_n\}_{n\in\w}\big)$ if $E$ is locally complete). Since $\cacx\big(\{x_n\}_{n\in\w}\big)\subseteq \cacx(A)$ is a coarse $p$-limited set (see (ii) of Lemma 
4.1 of \cite{Gab-limited}), it follows that  the operator $T$ is coarse $p$-limited. Therefore, by (iv) or (v), the sequence  $\{x_n\}_{n\in\w}\subseteq T(B_{\ell_1^0})$ has a convergent subsequence (resp., a Cauchy subsequence or $T\big(\{e_n\}_{n\in\w}\big)=\{x_n\}_{n\in\w}$ is precompact). Thus $A$ is relatively sequentially compact (resp., sequentially precompact or precompact), as desired.\qed
\end{proof}


\section{$p$-Gelfand--Phillips sequentially compact property of order $(q',q)$} \label{sec:GPscP}


We start this section with the next assertion. Recall (see Definition \ref{def:p-GP-q}) that an lcs $(E,\tau)$ is said to have the $p$-$GPscP_{(q',q)}$ if every weakly $p$-summable sequence in $E$ which is also a $(q',q)$-limited set is $\tau$-null.

\begin{proposition} \label{p:GPpq=>p-GPsc}
Let $1\leq q'\leq q\leq\infty$, and let $E$ be a locally convex space.
\begin{enumerate}
\item[{\rm(i)}] If $E$ has the $prGP_{(q',q)}$, then it has the $p$-$GPscP_{(q',q)}$ for every $p\in[1,\infty]$.
\item[{\rm(ii)}] If $q'\geq 2$ and $E$ has the Rosenthal property, then $E$ has the $prGP_{(q',q)}$ if and only if it has the $GPscP_{(q',q)}$.
\end{enumerate}
\end{proposition}

\begin{proof}
(i) follows from (i) of Theorem \ref{t:char-coarse-prGPp}, and (ii) follows from (iii) of Theorem \ref{t:char-coarse-prGPp}.\qed
\end{proof}

\begin{remark} \label{rem:GPsc-not-prGP} {\em
The converse in (i) of Proposition \ref{p:GPpq=>p-GPsc} is not true in general. Indeed, let $1<r,p<2$ and $E=\ell_r$. By Corollary \ref{c:wsGP-Schur-con}, $\ell_r$ does not have the $prGP_{(1,\infty)}$. On the other hand, let $\{x_n\}_{n\in\w}$ be a weakly $p$-summable sequence in $E$ which is also a $(q',q)$-limited set, and suppose for a contradiction that $x_n\not\to 0$. Passing to a subsequence we can assume that $1/c \leq \|x_n\| \leq c$ for some $c>1$ and all $n\in\w$. Dividing $x_n$ by its norm we can assume also that $\{x_n\}_{n\in\w}$ is a normalized sequence. Taking into account that  $\{x_n\}_{n\in\w}$ is weakly null,  Proposition 2.1.3 of \cite{Al-Kal} implies that there is a subsequence $\{x_{n_k}\}_{k\in\w}$ which is a basic sequence equivalent to the canonical basis of $\ell_r$ and such that $\cspn\big(\{x_{n_k}\}_{k\in\w}\big)$ is complemented in $\ell_r$. It follows that  the canonical basis  $\{e_n\}_{n\in\w}$ of $\ell_r$ is weakly $p$-summable that contradicts Example 
 4.4 of \cite{Gab-Pel} since $p<r^\ast$. Thus $x_n\to 0$ and $\ell_r$ has the $p$-$GPscP_{(q',q)}$ for all $1\leq q'\leq q\leq\infty$.\qed}
\end{remark}

Let $p,q\in[1,\infty]$. Following \cite{Gab-DP}, a locally convex space $E$ is said to have the {\em sequential Dunford--Pettis property of order $(p,q)$} (the {\em sequential $DP_{(p,q)}$ property}) if $\lim_{n\to \infty} \langle\chi_n,x_n\rangle=0$ for every weakly $p$-summable sequence $\{ x_n\}_{n\in\w}$ in $E$ and each weakly $q$-summable sequence $\{ \chi_n\}_{n\in\w}$ in  $E'_\beta$. Following \cite{Gab-p-Oper}, a linear map $T$ from $E$ to a locally convex space $L$ is called {\em weakly $(p,q)$-convergent}  if $\lim_{n\to\infty} \langle\eta_n, T(x_n)\rangle=0$ for every weakly $q$-summable sequence $\{\eta_n\}_{n\in\w}$ in $L'_\beta$ and each weakly $p$-summable sequence $\{x_n\}_{n\in\w}$ in $E$.
The following proposition, which will be used below, shows that the sequential $DP_{(p,q)}$ property of the range  implies some strong additional properties of operators.
\begin{proposition} \label{p:weak-pq-conver-L}
Let $p,q\in[1,\infty]$, and let $E$ and $L$ be  locally convex spaces. If $L$ has the sequential $DP_{(p,q)}$ property {\rm(}for instance, $L$ is a quasibarrelled locally complete space with the Dunford--Pettis property, e.g.  $L=\ell_\infty${\rm)}, then each operator $T\in\LL(E,L)$ is weakly $(p,q)$-convergent.
\end{proposition}

\begin{proof}
Let $\{\eta_n\}_{n\in\w}$ be a  weakly $q$-summable sequence in $L'_\beta$,  and let $\{x_n\}_{n\in\w}$ be a  weakly $p$-summable sequence  in $E$. Then $\{T(x_n)\}_{n\in\w}$ is a  weakly $p$-summable sequence  in $L$. Now the sequential $DP_{(p,q)}$ property of $L$ implies $\langle\eta_n,T(x_n)\rangle\to 0$. Thus $T$  is weakly $(p,q)$-convergent.

If $L$ is a quasibarrelled locally complete space with the Dunford--Pettis property, Corollary 
5.13 of \cite{Gab-DP} implies that $L$ has the sequential $DP_{(p,q)}$ property.\qed
\end{proof}

The following theorem characterizes locally convex spaces $E$ with the $p$-$GPscP_{(q',q)}$ and explains our usage of ``sequentially compact'' in the notion of the $p$-$GPscP_{(q',q)}$.
\begin{theorem} \label{t:char-p-GP}
Let $p,q,q'\in[1,\infty]$,  $q'\leq q$. Then for a locally convex space $(E,\tau)$, the following assertions are equivalent:
\begin{enumerate}
\item[{\rm(i)}] $E$ has the $p$-$GPscP_{(q',q)}$;
\item[{\rm(ii)}] each operator $T\in\LL(E,\ell_\infty)$ is $(q',q)$-limited $p$-convergent;
\item[{\rm(iii)}] each weakly $(p,q)$-convergent operator $T\in\LL(E,\ell_\infty)$ is $(q',q)$-limited $p$-convergent;
\item[{\rm(iv)}] each weakly sequentially $p$-compact subset $A$ of $E$ which is a $(q',q)$-limited set is relatively sequentially compact in $E$;
\item[{\rm(v)}] each weakly sequentially $p$-precompact subset $A$ of $E$ which is a $(q',q)$-limited set is sequentially precompact in $E$.
\end{enumerate}
\end{theorem}

\begin{proof}
(i)$\Ra$(ii) Assume that $E$ has the $p$-$GPscP_{(q',q)}$, and let $\{x_n\}_{n\in\w}$ be weakly $p$-summable sequence in $E$ which is a $(q',q)$-limited set. By the $p$-$GPscP_{(q',q)}$, we have $x_n\to 0$ in $E$. Therefore also $T(x_n)\to 0$ in $\ell_\infty$. Thus $T$ is a $(q',q)$-limited $p$-convergent operator.
\smallskip

(ii)$\Ra$(i) Assume that each operator $T\in\LL(E,\ell_\infty)$ is $(q',q)$-limited $p$-convergent. To show that $E$ has  the $p$-$GPscP_{(q',q)}$, let $\{x_n\}_{n\in\w}$ be a weakly $p$-summable sequence in $E$ which is a $(q',q)$-limited set. Assuming that $x_n\not\to 0$ in $E$ and  passing to a subsequence if needed we can assume that there is $U\in\Nn_0^c(E)$ such that $x_n\not\in U$ for all $n\in\w$. For every  $n\in\w$, choose $\chi_n\in U^\circ$ such that $\langle\chi_n,x_n\rangle>1$. Define a linear map $T:E\to \ell_\infty$ by
\[
T(x):= \big(\langle\chi_n,x\rangle\big)_n \quad (x\in E).
\]
Since $T(U)\subseteq B_{\ell_\infty}$, the map $T$ is continuous. As $\|T(x_n)\|_{\ell_\infty} \geq \langle\chi_n,x_n\rangle>1 \not\to 0$, it follows that $T$ is not $(q',q)$-limited $p$-convergent. This is a contradiction.
\smallskip

The equivalence (ii)$\Leftrightarrow$(iii) immediately follows from Proposition \ref{p:weak-pq-conver-L}.
\smallskip

(i)$\Ra$(iv) To show that $A$ is relatively sequentially compact in $E$, take an arbitrary sequence $S=\{a_n\}_{n\in\w}$ in $A$. Since $A$ is weakly sequentially $p$-compact, passing to a subsequence we can assume that $S$ weakly $p$-converges to some point $a\in A$. So the sequence $\{a_n-a\}_{n\in\w}$ is weakly $p$-summable and also a $(q',q)$-limited set. Therefore, by the $p$-$GPscP_{(q',q)}$, we have $a_{n} -a\to 0$. Thus $A$ is a  relatively sequentially compact subset of $E$.
\smallskip

(iv)$\Ra$(i) If $S=\{a_n\}_{n\in\w}$ is a weakly $p$-summable sequence in $E$ which is a  $(q',q)$-limited set, then $S$ is relatively sequentially compact in $E$. Being also weakly null $S$ is a null sequence in $E$ by Lemma \ref{l:null-seq}.
\smallskip

(i)$\Ra$(v) To show that $A$ is sequentially precompact in $E$, take an arbitrary sequence $S=\{a_n\}_{n\in\w}$ in $A$. Since $A$ is weakly sequentially $p$-precompact, we can assume that $S$ is weakly $p$-Cauchy. For every strictly increasing sequence $(n_k)$ in $\w$, it follows that $\{a_{n_k} -a_{n_{k+1}}\}_{k\in\w}$ is weakly $p$-summable and also a $(q',q)$-limited set. Therefore, by the $p$-$GPscP_{(q',q)}$, we have $a_{n_k} -a_{n_{k+1}}\to 0$ and hence $S$ is Cauchy in $E$. Thus $A$ is a sequentially precompact subset of $E$.
\smallskip

(v)$\Ra$(i) If $S=\{a_n\}_{n\in\w}$ is a weakly $p$-summable sequence in $E$ which is a  $(q',q)$-limited set, then $S$ is sequentially precompact in $E$. As $S$ is also weakly null, we apply Lemma \ref{l:null-seq} to get that $S$ is a null sequence in $E$.\qed
\end{proof}

Theorem \ref{t:char-p-GP} motivates the problem to characterize $(q',q)$-limited $p$-convergent operators. Under some additional restrictions this is done in the next assertion. If $E$ and $L$ are Banach spaces and $q'=q=\infty$, the following theorem is proved in Theorem 2.1 of \cite{SaMo} and Theorem 1.1 of \cite{WenChen}.
\begin{theorem} \label{t:char-q-lim-p-operator}
Let $2\leq q'\leq q\leq\infty$, and let $E$ be a locally convex space with the Rosenthal property.  Then for an operator $T$ from $E$ to a locally convex space $L$, the following assertions are equivalent:
\begin{enumerate}
\item[{\rm(i)}] $T$ is $(q',q)$-limited $\infty$-convergent;
\item[{\rm(ii)}]  for each  $(q',q)$-limited set $A\subseteq E$, the image $T(A)$ is sequentially precompact in $L$;
\item[{\rm(iii)}] for every locally convex {\rm(}the same, normed{\rm)} space  $H$ and each $(q',q)$-limited operator $R:H\to E$, the operator $T\circ R$ is sequentially precompact;
\item[{\rm(iv)}] for each $(q',q)$-limited operator $R:\ell_1^0\to E$, the operator $T\circ R$ is sequentially precompact.
\end{enumerate}
If in addition $E$ is locally complete, then {\rm(i)--(iv)} are equivalent to the following:
\begin{enumerate}
\item[{\rm(v)}] for each $(q',q)$-limited operator $R:\ell_1\to E$, the operator $T\circ R$ is sequentially precompact.
\end{enumerate}
\end{theorem}

\begin{proof}
(i)$\Ra$(ii) Assume that $T\in\LL(E,L)$ is $(q',q)$-limited $\infty$-convergent, and let $A$ be a $(q',q)$-limited subset of $E$. To show that $T(A)$ is sequentially precompact in $L$, let $\{a_n\}_{n\in\w}$ be a sequence in $A$. Since, by Theorem \ref{t:inf-V*-wsc}, $A$ is weakly sequentially precompact we can assume that $\{a_n\}_{n\in\w}$ is weakly Cauchy. Therefore, for every strictly increasing sequence $\{n_k\}_{k\in\w}$ in $\w$, the sequence $\{a_{n_{k+1}} -a_{n_k}\}_{k\in\w}$ is weakly null and also a $(q',q)$-limited set. Since $T$ is $(q',q)$-limited $\infty$-convergent it follows that $T(a_{n_{k+1}}) -T(a_{n_k})\to 0$ in $L$, and hence the sequence  $\{T(a_{n_k})\}_{k\in\w}$ is Cauchy in $L$. Thus $T(A)$ is sequentially precompact in $E$.


(ii)$\Ra$(iii) Take $U\in\Nn_0(H)$ such that $R(U)$ is a $(q',q)$-limited subset of $E$. By (ii), we have $T(R(U))$ is  sequentially precompact in $L$. Thus $T\circ R$ is a sequentially precompact operator.

(iii)$\Ra$(iv) and (iii)$\Ra$(v) are obvious.

(iv)$\Ra$(i) and (v)$\Ra$(i): Let $S=\{x_n\}_{n\in\w}$  be a weakly null sequence in $E$ which is a $(q',q)$-limited set. Since $S$ is bounded,
 Proposition 
14.9 of \cite{Gab-Pel} implies that the linear map $R:\ell_1^0 \to E$ (or $R:\ell_1 \to E$ if $E$ is locally complete) defined by
\[
R(a_0 e_0+\cdots+a_ne_n):=a_0 x_0+\cdots+ a_n x_n \quad (n\in\w, \; a_0,\dots,a_n\in\IF).
\]
is continuous.  It is clear that $R(B_{\ell_1^0})\subseteq \cacx\big(S\big)$ (or $R(B_{\ell_1})\subseteq \cacx\big(S\big)$ if $E$ is locally complete). Observe that $\cacx\big(S\big)$ is also a $(q',q)$-limited set. Therefore the operator $R$ is $(q',q)$-limited, and hence, by (iv) or (v), $T\circ R$ is sequentially precompact. In particular, the weakly null sequence $T(S)$ is (sequentially) precompact in $L$. Therefore, by Lemma \ref{l:null-seq}, 
$T(x_n)\to 0$ in $L$. Thus $T$ is $(q',q)$-limited $\infty$-convergent.\qed
\end{proof}

Below we obtain a sufficient condition on locally convex spaces to have the $p$-$GPscP_{(q,\infty)}$.
\begin{proposition} \label{p:char-p-GPscPp}
Let $p,q\in[1,\infty]$, and let $E$ be a locally convex space. Then $E$ has the $p$-$GPscP_{(q,\infty)}$ if one of the the following conditions holds:
\begin{enumerate}
\item[{\rm(i)}] for every locally convex space $X$ and for each $T\in\LL(X,E)$  whose adjoint $T^\ast:E'_{w^\ast}\to X'_\beta$ is $q$-convergent,  the operator $T$ transforms the bounded subsets of $X$ into sequentially precompact subsets of $E$;
\item[{\rm(ii)}] for every normed space $X$ and for each $T\in\LL(X,E)$ whose adjoint $T^\ast:E'_{w^\ast}\to X'_\beta$ is $q$-convergent,  the operator $T$ is  sequentially precompact;
\item[{\rm(iii)}] the same as {\rm(ii)} with $X=\ell_1^0$;
\item[{\rm(iv)}] if $E$ is locally complete, the same as {\rm(ii)} with $X$ a Banach space;
\item[{\rm(v)}]  if $E$ is locally complete, the same as {\rm(ii)} with $X=\ell_1$;
\end{enumerate}
\end{proposition}

\begin{proof}
Since the implications (i)$\Ra$(ii)$\Ra$(iii) and (ii)$\Ra$(iv)$\Ra$(v) are obvious, we shall prove that (iii) and (v) imply the $p$-$GPscP_{(q,\infty)}$. Suppose for a contradiction that $E$ has no the $p$-$GPscP_{(q,\infty)}$. Then there is a weakly $p$-summable sequence $S=\{x_n\}_{n\in\w}$ in $E$ which is a $(q,\infty)$-limited set such that $x_n\not\to 0$ in $E$. Without loss of generality we assume that $x_n \not\in U$ for some $U\in\Nn_0(E)$ and all $n\in\w$.

Since $S$ is a $(q,\infty)$-limited  set, Proposition 
5.9 of \cite{Gab-limited} implies that the linear map  $T:\ell_1^0 \to E$  (or  $T:\ell_1 \to E$ if $E$ is locally complete) defined by
\[
T(a_0 e_0+\cdots+a_ne_n):=a_0 x_0+\cdots+ a_n x_n \quad (n\in\w, \; a_0,\dots,a_n\in\IF),
\]
is continuous and its adjoint $T^\ast: E'_{w^\ast}\to\ell_\infty$ is $q$-convergent. Therefore, by (iii) or (v), the operator $T$ is  sequentially  precompact. In particular, the set $T(\{e_n\}_{n\in\w})=S$ is  (sequentially) precompact in $E$. Since $S$ is also weakly null, Lemma \ref{l:null-seq}
 implies that  $x_n\to 0$ in $E$. But this contradicts the choice of $S$.\qed
\end{proof}

For an important case which covers all strict $(LF)$ spaces we have the following characterization the $GPscP_{(q,\infty)}$.

\begin{theorem} \label{t:char-p-GPsc}
Let $2\leq q\leq\infty$. Then for a locally convex space $E$ with the Rosenthal property, the following assertions are equivalent:
\begin{enumerate}
\item[{\rm(i)}] $E$ has the $GPscP_{(q,\infty)}$;
\item[{\rm(ii)}] for every locally convex space $X$ and for each $T\in\LL(X,E)$  whose adjoint $T^\ast:E'_{w^\ast}\to X'_\beta$ is $q$-convergent,  the operator $T$ transforms the bounded subsets of $X$ into sequentially precompact subsets of $E$;
\item[{\rm(iii)}] for every normed space $X$ and for each $T\in\LL(X,E)$ whose adjoint $T^\ast:E'_{w^\ast}\to X'_\beta$ is $q$-convergent,  the operator $T$ is  sequentially precompact;
\item[{\rm(iv)}] the same as {\rm(ii)} with $X=\ell_1^0$;
\end{enumerate}
If in addition $E$ is locally complete, then {\rm(i)--(iv)} are equivalent to
\begin{enumerate}
\item[{\rm(v)}] the same as {\rm(iii)} with $X$ a Banach space;
\item[{\rm(vi)}] the same as {\rm(iii)} with $X=\ell_1$;
\end{enumerate}
\end{theorem}

\begin{proof}
By (the proof of)  Proposition \ref{p:char-p-GPscPp}, we have to prove only the implication (i)$\Ra$(ii).
So, let $X$ be a locally convex space, and let $T\in\LL(X,E)$ be such that the adjoint $T^\ast:E'_{w^\ast}\to X'_\beta$ is $q$-convergent. We have to show that for every bounded subset $B$ of $X$,  the set $T(B)$ is  sequentially precompact in $E$. To this end, let $\{\chi_n\}_{n\in\w}$ be a weak$^\ast$ $q$-summable sequence in $E'$. As $T^\ast$  is $q$-convergent, we have $ T^\ast(\chi_n)\to 0$ in $X'_\beta$. Since $B$ is bounded, the set $B^\circ$ is a neighborhood of zero in $X'_\beta$. Therefore, for every $\e>0$ there exists $N\in\w$ such that $T^\ast(\chi_n)\in \e B^\circ$ for all $n\geq N$. Now for every $y\in B$, we have
\[
|\langle\chi_n, T(y)\rangle|=|\langle T^\ast(\chi_n), y\rangle|\leq \e \;\; \mbox{ for all $n\geq N$}.
\]
Therefore $T(B)$ is a $(q,\infty)$-limited set in $E$. Since $q\geq 2$ and $E$ has the Rosenthal property, Theorem \ref{t:inf-V*-wsc} implies that $T(B)$ is weakly sequentially precompact. To show that $T(B)$ is sequentially precompact in $E$ it suffices to prove that for any sequence $S=\{x_n\}_{n\in\w}$ in $B$, the image $T(S)$ has a Cauchy subsequence. To this end, taking into account that $T(B)$ is weakly sequentially precompact, we can assume that $T(S)$ is weakly Cauchy. For any strictly increasing sequence $\{n_k\}_{k\in\w}$ in $\w$, the sequence $\{T(x_{n_k})-T(x_{n_{k+1}})\}_{k\in\w}$ is weakly null and a $(q,\infty)$-limited set. Therefore, by the $GPscP_{(q,\infty)}$ of $E$, we obtain $T(x_{n_k})-T(x_{n_{k+1}})\to 0$ in $E$. So $\{T(x_{n_k})\}_{k\in\w}$ is Cauchy in $E$.\qed
\end{proof}

Now we show that the class of spaces with the $p$-$GPscP_{(q',q)}$ has nice stability properties.

\begin{proposition} \label{p:subspace-pGPscP}
Let $p,q,q'\in[1,\infty]$,  $q'\leq q$, and let $H$ be a subspace of a locally convex space $E$. If $E$ has the $p$-$GPscP_{(q',q)}$, then also $H$ has  the $p$-$GPscP_{(q',q)}$.
\end{proposition}

\begin{proof}
Let $S=\{h_n\}_{n\in\w}$ be a weakly $p$-summable sequence in $H$ which is also a $(q',q)$-limited set. Then, by (vi) of Lemma 
4.6 of \cite{Gab-Pel}, $S$ is  weakly $p$-summable also in $E$, and, by (iv) of Lemma 
3.1 of \cite{Gab-limited}, $S$ is a $(q',q)$-limited subset of $E$. By the $p$-$GPscP_{(q',q)}$ of $E$, we have $h_n\to 0$ in $E$, and hence $h_n\to 0$ also in $H$. Thus $H$ has the $p$-$GPscP_{(q',q)}$.\qed
\end{proof}

\begin{proposition} \label{p:product-pGPscP}
Let $p,q,q'\in[1,\infty]$,  $q'\leq q$, and let $\{E_i\}_{i\in I}$ be a mom-empty family of locally convex spaces. Then the following assertions are equivalent:
\begin{enumerate}
\item[{\rm(i)}] the space $E=\prod_{i\in I} E_i$ has the $p$-$GPscP_{(q',q)}$;
\item[{\rm(ii)}] the space $L=\bigoplus_{i\in I} E_i$ has the $p$-$GPscP_{(q',q)}$;
\item[{\rm(iii)}] for every $i\in I$, the space $E_i$ has the $p$-$GPscP_{(q',q)}$.
\end{enumerate}
\end{proposition}

\begin{proof}
(i)$\Ra$(ii) Let $S=\{x_n\}_{n\in\w}$ be  a weakly $p$-summable sequence in $L$ which is also a $(q',q)$-limited set. Since $S$ is bounded, there is a finite subset $F$ of $I$ such that $S\subseteq \prod_{i\in F} E_i \times \prod_{i\in I\SM F} \{0_i\}$. Since $\prod_{i\in F} E_i$ is a subspace of $E$, Proposition \ref{p:subspace-pGPscP} implies that $x_n\to 0$ in $L$. Thus $L$ has the $p$-$GPscP_{(q',q)}$.
\smallskip

(ii)$\Ra$(iii) follows from Proposition \ref{p:subspace-pGPscP} because any $E_i$ can be considered as a subspace of $L$.
\smallskip

(iii)$\Ra$(i) Let $S=\{x_n=(x_{i,n})\}_{n\in\w}$ be  a weakly $p$-summable sequence in $E$ which is also a $(q',q)$-limited set. Observe that $x_n\to 0$ in $E$ if and only if $x_{i,n}\to 0_i$ for every $i\in I$. For every $i\in I$, (ii) of Lemma 
4.25 of \cite{Gab-Pel} and (i) of Proposition 
3.3 of \cite{Gab-limited} imply that the sequence $\{x_{i,n}\}_{n\in\w}$ is weakly $p$-summable in $E_i$ which is also a $(q',q)$-limited set. Therefore, by  the $p$-$GPscP_{(q',q)}$ of $E_i$, we obtain $x_{i,n}\to 0_i$ in $E_i$, as desired.\qed
\end{proof}


\section{Strong versions of the $V^\ast$ type properties of Pe{\l}czy\'{n}ski} \label{sec:strong-V*}


If in Definition \ref{def:property-Vp*} of $V^\ast$ type sets to ask that the $(p,q)$-$(V^\ast)$ sets are relatively compact in the {\em original} topology instead of the {\em weak} topology we obtain ``strong'' versions of $V^\ast_{(p,q)}$ properties.

\begin{definition} \label{def:V*-limited} {\em
Let $1\leq p\leq q\leq \infty$. A locally convex space $(E,\tau)$ is said to have
\begin{enumerate}
\item[$\bullet$] the {\em strong $V^\ast_{(p,q)}$ property} (resp., {\em strong $EV^\ast_{(p,q)}$ property}) if every $(p,q)$-$(V^\ast)$ (resp., $(p,q)$-$(EV^\ast)$) subset  of $E$ is relatively compact;
\item[$\bullet$] the {\em strong sequential $V^\ast_{(p,q)}$ property}  (resp., {\em strong  sequential $EV^\ast_{(p,q)}$ property}) if every $(p,q)$-$(V^\ast)$ (resp., $(p,q)$-$(EV^\ast)$) subset of $E$ is relatively sequentially compact; for short, the {\em strong  $sV^\ast_{(p,q)}$ property} or the {\em strong  $sEV^\ast_{(p,q)}$ property};
\item[$\bullet$] the {\em strong precompact $V^\ast_{(p,q)}$ property} (resp., {\em strong precompact $EV^\ast_{(p,q)}$ property})  if every $(p,q)$-$(V^\ast)$ (resp.,  $(p,q)$-$(EV^\ast)$)  subset of $E$ is $\tau$-precompact; for short, the {\em strong $prV^\ast_{(p,q)}$ property} or the {\em strong $prEV^\ast_{(p,q)}$ property};
\item[$\bullet$] the {\em strong sequentially precompact $V^\ast_{(p,q)}$ property} (resp., {\em strong sequentially precompact $EV^\ast_{(p,q)}$ property})  if every $(p,q)$-$(V^\ast)$ (resp.,  $(p,q)$-$(EV^\ast)$)  subset of $E$ is sequentially precompact in $\tau$; for short, the {\em strong $spV^\ast_{(p,q)}$ property} or the {\em strong $spEV^\ast_{(p,q)}$ property}.
\end{enumerate}
If $q=\infty$, we shall say that $E$ has the strong (resp., strong sequential, precompact or sequentially precompact) $V^\ast_{p}$ property. If $p=1$ and $q=\infty$, then we say that $E$ has the strong  (resp., strong sequential, precompact or sequentially precompact) $V^\ast$ property. \qed}
\end{definition}

The next lemma follows from the corresponding definitions and the simple fact that every $(p,q)$-limited set is also a $(p,q)$-$(V^\ast)$ set.
\begin{lemma} \label{l:strong-V*-property}
Let $1\leq p\leq q\leq \infty$, and let $E$ be a locally convex space. Then:
\begin{enumerate}
\item[{\rm(i)}] if $E$ has the strong $V^\ast_{(p,q)}$ property, then it has the $GP_{(p,q)}$ property;
\item[{\rm(ii)}] if $E$ has the strong $sV^\ast_{(p,q)}$ property, then it has the $sGP_{(p,q)}$ property;
\item[{\rm(iii)}] if $E$ has the strong $prV^\ast_{(p,q)}$ property, then it has the $prGP_{(p,q)}$ property;
\item[{\rm(iv)}] if $E=E_w$, then $E$  has the strong $prV^\ast_{(p,q)}$ property.
\end{enumerate}
\end{lemma}

\begin{remark} {\em
The converse in Lemma \ref{l:strong-V*-property} is not true in general even for Banach spaces. Indeed, let $p=q=\infty$. Then, by Theorem \ref{t:Banach-GP}, $c_0$ has the Gelfand--Phillips property. However, $c_0$ does not have the strong $V^\ast_\infty$ property because $B_{c_0}$ is not compact but it is an $\infty$-$(V^\ast)$ set since any weakly null sequence in $\ell_1=(c_0)'_\beta$, by the Schur property, is norm null.\qed}
\end{remark}

The next proposition shows that for a wide class of locally convex spaces including all strict $(LF)$-spaces all strong $V^\ast$ type properties coincide.
\begin{proposition} \label{p:LF-strong-V*}
Let $1\leq p\leq q\leq \infty$, and let $E$ be a von Neumann complete, $p$-quasibarrelled angelic space {\rm(}for example, $E$ is a strict $(LF)$-space{\rm)}. Then all the properties of being strong $V^\ast_{(p,q)}$, strong $EV^\ast_{(p,q)}$, strong $sV^\ast_{(p,q)}$, strong  $sEV^\ast_{(p,q)}$, strong $prV^\ast_{(p,q)}$, strong $prEV^\ast_{(p,q)}$, strong  $spV^\ast_{(p,q)}$ or  strong  $spEV^\ast_{(p,q)}$ are equivalent.
\end{proposition}

\begin{proof}
Since $E$ is $p$-quasibarrelled, Lemma 
7.2 of \cite{Gab-Pel} implies that the family of $(p,q)$-$(V^\ast)$ sets coincides with the family of $(p,q)$-$(EV^\ast)$ sets. Therefore all equicontinuous $V^\ast$ type properties coincide with the corresponding $V^\ast$ type property. Since $E$ is von Neumann complete, it is clear that the strong $prV^\ast_{(p,q)}$ property coincides with the strong $V^\ast_{(p,q)}$ property. Finally, the angelicity of $E$ implies that the (relatively) compact subsets of $E$ are exactly the (relatively) sequentially compact sets. Taking into account that the closure of a $(p,q)$-$(V^\ast)$ set is a $(p,q)$-$(V^\ast)$ set, it follows that the strong $V^\ast_{(p,q)}$ (resp., $prV^\ast_{(p,q)}$) property coincides with  the strong $sV^\ast_{(p,q)}$ (resp., strong $spV^\ast_{(p,q)}$) property. \qed
\end{proof}

The following two propositions show that the classes of locally convex spaces with strong $V^\ast_{(p,q)}$ type properties are stable under taking direct products, direct sums and closed subspaces. We omit their proofs because they can be obtained from the proofs of Propositions \ref{p:product-sum-GP} and \ref{p:subspace-pGP}, respectively, just replacing ``$(p,q)$-limited'' by ``$(p,q)$-$(V^\ast)$'' and from Proposition 
7.4 of \cite{Gab-Pel} which describes $(p,q)$-$(V^\ast)$ sets in direct products and direct sums.
\begin{proposition} \label{p:product-sum-V*-strong}
Let $1\leq p\leq q\leq\infty$, and let $\{E_i\}_{i\in I}$  be a non-empty family of locally convex spaces. Then:
\begin{enumerate}
\item[{\rm(i)}] $E=\prod_{i\in I} E_i$ has the strong $V^\ast_{(p,q)}$ $($resp., strong $EV^\ast_{(p,q)}$, strong $prV^\ast_{(p,q)}$ or strong  $prEV^\ast_{(p,q)}$$)$ property if and only if all spaces $E_i$ have the same property;
\item[{\rm(ii)}] $E=\bigoplus_{i\in I} E_i$ has the strong $V^\ast_{(p,q)}$ $($resp., strong $EV^\ast_{(p,q)}$, strong $sV^\ast_{(p,q)}$, strong $sEV^\ast_{(p,q)}$,  strong $prV^\ast_{(p,q)}$, strong $prEV^\ast_{(p,q)}$, strong $spV^\ast_{(p,q)}$ or strong $spEV^\ast_{(p,q)}$$)$ property if and only if all spaces  $E_i$ have  the same property;
\item[{\rm(iii)}] if $I=\w$ is countable, then $E=\prod_{i\in \w} E_i$ has the  strong $sV^\ast_{(p,q)}$ $($resp.,   strong $sEV^\ast_{(p,q)}$, strong $spV^\ast_{(p,q)}$ or strong $spEV^\ast_{(p,q)}$$)$ property  if and only if all spaces $E_i$ have the same property.
\end{enumerate}
\end{proposition}

\begin{proposition} \label{p:subspace-V*-strong}
Let $1\leq p\leq q\leq\infty$, and let $L$ be a subspace of a locally convex space $E$.
\begin{enumerate}
\item[{\rm(i)}] If $L$ is closed in $E$ and $E$ has  the  the strong $V^\ast_{(p,q)}$ $($resp., strong $EV^\ast_{(p,q)}$, strong $prV^\ast_{(p,q)}$ or strong $prEV^\ast_{(p,q)}$$)$ property, then also $L$ has the same property.
\item[{\rm(ii)}] If $L$ is sequentially closed in $E$ and $E$ has the  strong $sV^\ast_{(p,q)}$ $($resp.,   strong $sEV^\ast_{(p,q)}$$)$ property, then also $L$ has the same property.
\item[{\rm(iii)}]  If $E$ has  the strong $prV^\ast_{(p,q)}$ $($resp., strong $prEV^\ast_{(p,q)}$, strong $spV^\ast_{(p,q)}$ or strong $spEV^\ast_{(p,q)}$$)$ property,  then also $L$ has the same property.
\end{enumerate}
\end{proposition}

Below we characterize locally convex spaces with the strong precompact $V^\ast_p$ property and the strong sequentially precompact $V^\ast_p$ property.
\begin{theorem} \label{t:char-strong-V*p}
Let $p\in[1,\infty]$, and let $E$ be a locally convex space. Then the following assertions are equivalent:
\begin{enumerate}
\item[{\rm(i)}] $E$ has the strong {\rm(}resp., sequentially{\rm)} precompact $V^\ast_p$ property;
\item[{\rm(ii)}] for every locally convex space $X$, each $T\in\LL(X,E)$, whose adjoint $T^\ast:E'_\beta\to X'_\beta$ is $p$-convergent, transforms bounded sets of $X$ into {\rm(}resp., sequentially{\rm)} precompact subsets of $E$;
\item[{\rm(iii)}] for every normed space $X$ and each $T\in\LL(X,E)$, if the adjoint $T^\ast:E'_\beta\to X'_\beta$ is $p$-convergent then $T$ is  {\rm(}resp., sequentially{\rm)} precompact;
\item[{\rm(iv)}] the same as {\rm(iii)} with $X=\ell_1^0$.
\end{enumerate}
If in addition $E$ is locally complete, then {\rm(i)--(iv)} are equivalent to
\begin{enumerate}
\item[{\rm(v)}] the same as {\rm(iii)} with $X$ a Banach space;
\item[{\rm(vi)}] the same as {\rm(iii)} with $X=\ell_1$.
\end{enumerate}
\end{theorem}

\begin{proof}
(i)$\Ra$(ii) Let $X$ be a locally convex space, and let $T\in\LL(X,E)$ be such that the adjoint $T^\ast:E'_\beta\to X'_\beta$ is $p$-convergent. We have to show that the set $T(B)$ is  (sequentially) precompact in $E$ for every bounded subset $B$ of $X$. To this end, let $\{\chi_n\}_{n\in\w}$ be a weakly $p$-summable sequence in $E'_\beta$.  As $T^\ast$  is $p$-convergent, we have $ T^\ast(\chi_n)\to 0$ in $X'_\beta$.  Since $B$ is bounded, the set $B^\circ$ is a neighborhood of zero in $X'_\beta$. Therefore, for every $\e>0$ there exists $N\in\w$ such that $T^\ast(\chi_n)\in \e B^\circ$ for all $n\geq N$. Now for every $y\in B$, we have
\[
|\langle\chi_n, T(y)\rangle|=|\langle T^\ast(\chi_n), y\rangle|\leq \e \;\; \mbox{ for all $n\geq N$}.
\]
Therefore $T(B)$ is a $p$-$(V^\ast)$ set in $E$. By the strong (sequentially)  precompact $V^\ast_p$ property it follows that $T(B)$ is (sequentially) precompact in $E$, as desired.
\smallskip

(ii)$\Ra$(iii)$\Ra$(iv) and (iii)$\Ra$(v)$\Ra$(vi) are obvious.
\smallskip

(iv)$\Ra$(i) and (vi)$\Ra$(i): To show that $E$ has the strong  (sequentially) precompact $V^\ast_p$ property, we assume for a contradiction that there is a $p$-$(V^\ast)$ set $A$ in $E$ which is not  (resp., sequentially) precompact. Then there are $U\in\Nn_0(E)$ and a $U$-separated sequence $S=\{x_n\}_{n\in\w}$ in $A$ (resp., a sequence $S=\{x_n\}_{n\in\w}$ in $A$ which does not have a Cauchy subsequence). Since $S$ is also a $p$-$(V^\ast)$ set, Proposition 
14.9 of \cite{Gab-Pel} implies that the linear map  $T:\ell_1^0 \to E$  (or  $T:\ell_1 \to E$ if $E$ is locally complete) defined by
\[
T(a_0 e_0+\cdots+a_ne_n):=a_0 x_0+\cdots+ a_n x_n \quad (n\in\w, \; a_0,\dots,a_n\in\IF),
\]
is continuous and its adjoint $T^\ast: E'_\beta\to\ell_\infty$ is $p$-convergent. Therefore, by (iv) or (vi), the operator $T$ is  (resp., sequentially)  precompact. In particular, the set $T(\{e_n\}_{n\in\w})=S$ is  (sequentially) precompact in $E$ which contradicts the choice of $S$.\qed
\end{proof}

\bibliographystyle{amsplain}

\end{document}